\documentclass{article}
\usepackage[utf8]{inputenc}
\usepackage{fullpage}
\usepackage{amsmath}
\usepackage{amssymb}
\usepackage{tikz-cd}
\usepackage {amscd}
\usepackage{bbm}
\usepackage[all,cmtip]{xy}
\usepackage{amsfonts}
\usepackage{amsthm}
\usepackage{mathrsfs}
\usepackage{mathabx}
\usepackage{times}
\usepackage{hyperref}
\usepackage{enumitem}

\newtheorem{theorem}{Theorem}[section]

\newtheorem{definition}[theorem]{Definition}
\newtheorem{proposition}[theorem]{Proposition}
\newtheorem{lemma}[theorem]{Lemma}
\newtheorem{remark}[theorem]{Remark}
\newtheorem{assumption}{Assumption}

\newcommand{\mql}{\overline{\mathbb{Q}}_{\ell}^\times}
\newcommand{\ql}{\overline{\mathbb{Q}}_{\ell}}
\newcommand{\Fqcl}{\overline{\mathbb{F}}_{q}}
\newcommand{\dlt}{\widecheck{T}(\ql)}
\newcommand{\dt}{\widecheck{T}}
\newcommand{\Z}{\mathbb{Z}}
\newcommand{\Q}{\mathbb{Q}}
\newcommand{\G}{\mathbb{G}}
\newcommand{\K}{\breve{K}}
\newcommand{\E}{\breve{E}}
\newcommand{\F}{\mathbb{F}}
\newcommand{\N}{\Z_{>0}}
\newcommand{\T}{\mathcal{T}}
\newcommand{\Fq}{\mathbb{F}_q}
\renewcommand{\O}{\breve{\mathcal{O}}}
\newcommand{\Ocal}{\mathcal{O}}
\newcommand{\xto}{\xrightarrow}
\newcommand{\chg}{X^\ast}
\newcommand{\cchg}{X_\ast}

\newcommand{\wGamma}{\widetilde{\Gamma}}
\DeclareMathOperator{\cor}{Cor}
\DeclareMathOperator{\Hom}{Hom}
\DeclareMathOperator{\gal}{Gal}
\DeclareMathOperator{\im}{Im}
\DeclareMathOperator{\res}{Res}
\DeclareMathOperator{\nm}{Nm}
\DeclareMathOperator{\ab}{ab}

\DeclareMathOperator{\re}{r}
\DeclareMathOperator{\tr}{tr}
\DeclareMathOperator{\Tr}{Tr}
\DeclareMathOperator{\p}{p}
\newcommand{\tH}{\widehat{H}}
\newcommand{\elp}{L^{+}_{\E}}
\newcommand{\lp}{L^{+}}
\newcommand{\lpf}{L}
\newcommand{\CS}{\operatorname{CS}}
\DeclareMathOperator{\ch}{ch}
\DeclareMathOperator{\Fr}{Fr}
\DeclareMathOperator{\sm}{sm}
\DeclareMathOperator{\tran}{t}
\DeclareMathOperator{\Lang}{L}
\DeclareMathOperator{\opp}{op}
\newcommand{\onto}{\twoheadrightarrow}
\DeclareMathOperator{\ind}{Ind}
\DeclareMathOperator{\sh}{Sh}
\DeclareMathOperator{\val}{val}
\DeclareMathOperator{\infl}{Infl}

\renewcommand{\subset}{\subseteq}
\renewcommand{\supset}{\supseteq}

\title{Character Sheaves on Tori over Local Fields}
 \author{Tanmay Deshpande and Saniya Wagh}
\date{}

\begin{document}

\maketitle
\begin{abstract}
Let $\K$ be a complete discrete valuation field with an algebraically closed residue field ${k}$ and ring of integers $\O$. Let $T$ be a torus defined over $\K$. Let $L^+T$ denote the connected commutative pro-algebraic group over ${k}$ obtained by applying the Greenberg functor to the connected N\'eron model of $T$ over $\O$. Following the work of Serre for the multiplicative group, we first compute the fundamental group $\pi_1(L^+T)$. We then study multiplicative local systems (or character sheaves) on $L^+T$ and establish a local Langlands correspondence for them. Namely, we construct a canonical isomorphism of abelian groups between the group of multiplicative local systems on $L^+T$ and inertial local Langlands parameters for $T$. Finally, we relate our results to the classical local Langlands correspondence for tori over local fields due to Langlands, via the sheaf-function correspondence.
\end{abstract}

\section{Introduction}
Suppose $K$ is a non-Archimedean local field, that is, a complete discrete valuation field with a finite residue field. Then the reciprocity map of local class field theory gives us a canonical identification 
\begin{equation}\label{rec}
\re:K^\times\xrightarrow{\cong} W_K^{\ab}\mbox{ of topological groups,}
\end{equation}
 where $W_K^{\ab}$ denotes the abelianization of the Weil group of $K$. This statement is essentially the local Langlands correspondence for the multiplicative group $\G_m$ defined over $K$. 
 
 More generally, let $T$ be any torus defined over the local field $K$ and let $\dt$ denote the dual torus defined over $\mathbb{C}$. Note that the dual torus $\dt$ comes equipped with an action of the Weil group $W_K$ factoring through a finite quotient $W_K\twoheadrightarrow \gal(E/K)$ where $E/K$ is a finite Galois extension over which the torus $T$ splits.  In \cite{Langlands1997}, Langlands established a canonical isomorphism 
 \begin{equation}\label{eq:LLCfortori}
     \phi:\Hom_{\sm}(T(K),\mathbb{C}^\times)\xrightarrow{\cong} H^1_{cts}(W_K,\dt(\mathbb{C}))
 \end{equation}
 between the abelian group of smooth characters of the locally profinite group $T(K)$ and the abelian group of local Langlands parameters for $T$, namely the group of continuous group cohomology classes $H^1_{cts}(W_K,\dt(\mathbb{C}))$. This is the local Langlands correspondence for general tori. Note that the local Langlands parameters can equivalently be thought of as $\dt(\mathbb{C})$-conjugacy classes of continuous group homomorphisms $$W_K\to { }^LT(\mathbb{C}):=\dt(\mathbb{C})\rtimes W_K$$
 compatible with the natural projection ${}^LT(\mathbb{C})\to W_K$. We also remark that both the sides in (\ref{eq:LLCfortori}) remain unchanged as abstract abelian groups, whether we take the analytic topology or the discrete topology on $\mathbb{C}$.

Our goal in this paper is to study a geometric analogue of the above correspondence. Let $\K$ be a complete discrete valuation field with an algebraically closed residue field $k$. Let $\Ocal_{\K}$ denote the ring of integers of $\K$, which we will often denote as simply $\O$. In this setting, Serre proved a geometric analogue of local class field theory in \cite{Serre1961}. Namely, we may consider the group of units $\O^\times$ as a commutative pro-algebraic group $L^+\G_m$ defined over the residue field $k$, using the Greenberg functor applied to the multiplicative group $\G_m$ defined over $\O$. Serre defined the fundamental groups $\pi_1$ of such objects in \cite{Serre1960} and proved in \cite{Serre1961} that there is a canonical isomorphism
\begin{equation}\label{th}
    \theta:\pi_1(L^+\G_m)\xrightarrow{\cong} \gal^{\ab}_{\K}\mbox{ of topological groups,}
\end{equation}
where $\gal^{\ab}_{\K}$ denotes the abelianization of the absolute Galois group of $\K$.

Now let $T$ be any torus defined over $\K$. Let $\T$ denote its connected N\'eron model defined over the ring of integers $\O$. We can then apply the Greenberg functor, or the positive loops functor, and define the connected commutative pro-algebraic group $L^+T:=L^+\T$ defined over the algebraically closed residue field $k$. By definition, we have $L^+T(k)=\ \T(\O)$, which is the maximal connected schematic subgroup of $T(\K)$. Note that we will follow the conventions from \cite{Serre1960, Serre1961} about pro-algebraic groups over $k$. In particular, we will work in the setting of inverse systems of perfect schemes and perfect group schemes over the residue field $k$. In this paper, we will only be working with commutative pro-algebraic groups. Note that the group $L^+T$ is the neutral connected component of the full loop group $LT$ which can be thought of as a commutative group ind-scheme. In our first main result, we compute the fundamental group (in the sense of Serre \cite{Serre1960}) of the connected commutative pro-algebraic group $L^+T$. 

Let $\E/\K$ be a finite Galois extension over which the torus $T$ splits. Hence we have an action of $\gal(\E/\K)$ on the character lattice $X^*(T)$ and the co-character lattice  $X_*(T)$. We have the short exact sequence of Galois groups
\begin{equation}
    1\to \gal_{\E}\to \gal_{\K}\to \gal(\E/\K)\to 1
\end{equation}
which gives us an action of $\gal(\E/\K)$ on $\gal_{\E}^{\ab}$ and a transfer isomorphism (see \cite{Serre1961} for details)\begin{equation}
\gal_{\K}^{\ab}\cong \left(\gal_{\E}^{\ab}\right)^{\gal(\E/\K)}\hspace{-7pt}\subset \gal_{\E}^{\ab}.
\end{equation}
\begin{theorem}\label{thm:main1}
    Let $T$ be any torus defined over $\K$ as above. Then in the notation above, the fundamental group (in the sense of \cite{Serre1960}) of the commutative pro-algebraic group $L^+T$ defined over $k$ is given  by a canonical isomorphism 
    \begin{equation}\label{eq:pi_1computation}
        \pi_1(L^+T) \cong \left(X_*(T)\otimes\gal_{\E}^{\ab}\right)^{\gal(\E/\K)}.
    \end{equation}
\end{theorem}
A priori, the right hand side above seems to depend on the choice of the finite Galois extension $\E/\K$ which splits $T$. However, it is easy to see using the transfer isomorphisms that the right hand side of (\ref{eq:pi_1computation}) is independent of the choice of $\E$. A more precise form of this result is stated and proved as Theorem \ref{thm:fundamentalgroup} in \S\ref{sec:pfmain1}.

We now fix a prime number $\ell$ invertible in the residue field $k$. Our next goal is to study the multiplicative $\ql$-local systems, or character sheaves, on $L^+T$ and to establish a Langlands type correspondence for them. Multiplicative $\ql$-local systems on a connected commutative  pro-algebraic group $G$ over $k$ are in a natural bijection with continuous characters (see Section \ref{sec:CSonCCG} for more)
\[\chi:\pi_1(G)\to \mql\] 
of the Serre fundamental group $\pi_1(G)$, where we work with the $\ell$-adic topology on $\ql^\times$. 
We will use either of the notations below: \begin{equation}\label{eq:CSG}
    \CS(G)=\Hom_{\ell\text{-adic}}(\pi_1(G),\ql^\times),
\end{equation} to denote the abelian group of multiplicative $\ql$-local systems on any connected commutative pro-algebraic group $G$ defined over any algebraically closed field $k$. In particular, we are interested in studying the abelian group $\CS(L^+T)$.

Let us now describe the inertial $\ql$-Langlands parameters which form the other side of the Langlands correspondence for character sheaves on $L^+T$. Let $\dt$ denote the (split) dual torus defined over $\Z$. The absolute Galois group $\gal_{\K}$ acts on $\dt$ via the finite quotient $\gal_{\K}\twoheadrightarrow\gal(\E/\K)$. Consider the semi-direct product 
\begin{equation}
    {}^LT:=\dt\rtimes\gal_{\K}
\end{equation} 
as a group scheme over $\Z$. We let $\Hom_{\gal_{\K},\ell\text{-adic}}(\gal_{\K},{ }^LT(\ql))$ denote the set of continuous group homomorphisms $\gal_{\K}\to { }^LT(\ql)$ (in the $\ell$-adic topology) which are compatible with the natural projection ${}^LT\to \gal_{\K}$ as in the commutative diagram below:
\begin{equation}\label{eq:illpcommute}
    \xymatrix{
\gal_{\K}\ar@{=}[rd]\ar[rr] & & {}^LT(\ql).\ar@{->>}[ld]\\
 & \gal_{\K} &
 }
\end{equation}

We have an action of $\dt(\ql)$ on $\Hom_{\gal_{\K},\ell\text{-adic}}(\gal_{\K},{ }^LT(\ql))$ by conjugation and we define inertial local $\ql$-Langlands parameters to be the quotient $\Hom_{\gal_{\K},\ell\text{-adic}}\left.(\gal_{\K},{ }^LT(\ql))\middle/\dt(\ql)\right.$. We may also describe this set as the set of equivalence classes of continuous 1-cocycles $\gal_{\K}\to \dt(\ql)$, which we denote as $H^1_{\ell\text{-adic}}(\gal_{\K}, \dt(\ql)).$ It is then clear that these inertial $\ql$-Langlands parameters form an abelian group. We can now state our next main result which establishes an inertial local Langlands correspondence for $\ql$-character sheaves on $L^+T$:

\begin{theorem}\label{thm:main2}
    For a torus $T$ defined over $\K$ as  before, let $\CS^+(T)$ denote the abelian group of multiplicative $\ql$-local systems (or character sheaves) on the connected commutative pro-algebraic group $L^+T$. Then we have a canonical isomorphism of abelian groups
    \begin{equation}
        \CS^+(T)\cong H^1_{\ell\text{-adic}}(\gal_{\K}, \dt(\ql)).
    \end{equation}
\end{theorem}
We will state and prove a more refined version of this result in \S\ref{sec:pfmain2}. Once we have this inertial local Langlands correspondence for character sheaves on $L^+T$, we will relate it to the classical correspondence (\ref{eq:LLCfortori}) established by Langlands via the sheaf-function correspondence in \S\ref{sec:pfmain3}. We first recall the sheaf-function correspondence for connected commutative algebraic and pro-algebraic groups defined over finite fields due to Lusztig, cf. \cite[\S 5]{Lus06} and \cite[\S 1.5]{BD06}. We also refer to \S\ref{sec:CSonCCG} below for a more detailed exposition.

Let $k=\Fqcl$, where $q$ is a power of some prime number $p$. Let $G$ be a connected commutative algebraic group over $k$ equipped with an $\Fq$-structure defined by an $\Fq$-Frobenius endomorphism $\Fr:G\to G$. Then the Frobenius endomorphism permutes the set $\CS(G)$ of multiplicative local systems on $G$, and the sheaf-function correspondence gives us a canonical identification
\begin{equation}
    \Hom(G(\Fq),\ql^\times)\cong \CS(G)^{\Fr}
\end{equation} between the Pontryagin dual of the finite abelian group $G(\Fq)$ and $\Fr$-invariant multiplicative local systems on $G$. Using the same argument from \cite[\S 1.5]{BD06} we can extend the above result to connected commutative pro-algebraic groups $G$ defined over $\Fq$ and obtain a canonical identification of abelian groups
\begin{equation}
    \Hom_{\ell\text{-adic}}(G(\Fq),\ql^\times)\cong \CS(G)^{\Fr}
\end{equation} between the abelian groups of continuous $\ql$-characters of the commutative profinite group $G(\Fq)$ and $\Fr$-invariant character sheaves on $G$. However, in this paper, we will only need to deal with connected commutative pro-algebraic groups that satisfy some further simplifying assumptions. We will state and prove a more refined version of the above statement for this special class of pro-algebraic groups in \S\ref{sec:CSonCCG}.

To relate Theorem \ref{thm:main2} to (\ref{eq:LLCfortori}), we return to the setting of a torus $T$ defined over a non-Archimedean local field $K$ with ring of integers $\Ocal$ and a finite residue field $\Fq$. Let $\mathcal{T}$ denote the connected N\'eron $\Ocal$-model for $T$. Let $\K$ denote the completion of the maximal unramified extension of $K$ and $\O$ its ring of integers. The residue field of $\K$ is then the algebraically closed field $k=\Fqcl$. We have the short exact sequence $1\to \gal_{\K}\to W_K\to \Z\to 1$ of groups. The absolute Galois group $\gal_{\K}$  of $\K$ is often called the inertia subgroup of $W_K$. We have the Frobenius automorphism $\Fr=\Fr_K:\K\to \K$, which is a generator of the quotient $W_K/{\gal_{\K}}=\Z$.

By a slight abuse of notation, we continue to denote by $T$ the base change of the torus $T$ to $\K$. Moreover, the connected N\'eron $\O$-model for $T$ is simply the base change to $\O$ of the connected N\'eron $\Ocal$-model $\mathcal{T}$. Again, we continue to denote by $\mathcal{T}$  the base change of $\mathcal{T}$ to $\O$. We see that in the current setting, the positive loop group $L^+T$, which is defined over $k$, in fact has an $\Fq$-structure given by an $\Fq$-Frobenius endomorphism $\Fr: L^+T \to L^+T$ since the N\'eron model $\mathcal{T}$ is defined over $\Ocal$. Note that we have 
\[\mathcal{T}(\O)=L^+T(k)\mbox{ and }\mathcal{T}(\Ocal)=L^+T(\Fq).\]

The Kottwitz homomorphism for $T$ (cf. \cite{Kottwitz[1997]}) gives us the following two short exact sequences:
\begin{equation}
     0 \to \mathcal{T}(\O) \to T(\K) \to X_*(T)_{\gal_{\K}} \to 0
\end{equation} and taking Frobenius fixed points (and using Lang's Theorem \ref{thm:langprofin})
\begin{equation}\label{eq:kottwizforK}
     0 \to \mathcal{T}(\Ocal) \to T(K) \to \left(X_*(T)_{\gal_{\K}}\right)^{\Fr} \to 0.
\end{equation} We refer to \S\ref{sec:kottwitz} for a detailed discussion.

In the local Langlands correspondence for $T$, we are interested in parametrizing the smooth 1-dimensional characters
\[\chi:T(K)\to \ql^\times.\]
However since $T(K)/\mathcal{T}(\Ocal)$ is discrete and $\mathcal{T}(\Ocal)$ is profinite with a finite $\ell$-primary part, the smooth characters in fact coincide with $\ell$-adic (i.e. continuous with respect to the $\ell$-adic topology on $\ql^\times$) characters  for the groups $T(K),\mathcal{T}(\Ocal)$. We also obtain (see \S\ref{sec:pfmain3} for details) the dual exact sequence from (\ref{eq:kottwizforK})
\begin{equation}\label{eq:dualkottwitz}
    0\to\Hom((\cchg(T)_{\gal_{\K}})^{\Fr},\mql)\to\Hom_{\sm}(T(K),\mql)\to \Hom_{\sm}(\mathcal{T}(\Ocal),\mql)\to 0.
\end{equation}
Applying the results on the sheaf-function correspondence from \S\ref{sec:CSonCCG} to the connected commutative pro-algebraic group $L^+T$ and using Theorem \ref{thm:main2}, we obtain the isomorphism
\begin{equation}\label{eq:psidef}
    \psi:\Hom_{\sm}(\mathcal{T}(\mathcal{O}),\ql^\times)=\Hom_{\ell\text{-adic}}(L^+T(\Fq),\ql^\times)\xrightarrow{\cong} \CS^+(T)^{\Fr}\xrightarrow{\cong} H^1_{\ell\text{-adic}}(\gal_{\K}, \dt(\ql))^{\Fr}.
\end{equation}

For the other side of the Langlands correspondence, we have the dual torus $\dt$ (defined over $\Z$) equipped with an action of the Weil group $W_K$ which factors through a finite quotient $W_K\onto \gal(E/K)$, where $E/K$ is a finite Galois extension over which $T$ splits. We have the short exact sequence of topological groups
\[
1\to \gal_{\K} \to W_K\to \Z\to 1
\]
with the above action of $W_K$ on $\dt(\ql)$ and hence we have the associated inflation-restriction exact sequence. Our final main result below describes the relationship between Theorem \ref{thm:main2} and the classical local Langlands correspondence (\ref{eq:LLCfortori}) for tori.

\begin{theorem}\label{thm:main3}
Let $T$ be a torus defined over a non-Archimedean local field $K$. Then in the notation used above, we have the following two isomorphic short exact sequences of abelian groups fitting into a commutative diagram
\[
    \xymatrix{
0\ar[r]&\Hom((\cchg(T)_{\gal_{\K}})^{\Fr},\mql)\ar[r]\ar[d]_\cong&\Hom_{\sm}(T(K),\mql)\ar[r]^\res\ar[d]_\cong^\phi & \Hom_{\sm}(\mathcal{T}(\Ocal),\mql)\ar[r]\ar[d]_\cong^\psi & 0\\
  0\ar[r]& H^1(\Z,\dlt^{\gal_{\K}})\ar[r]^{\inf} & H_{\ql}^{1}(W_K,\dlt)\ar[r]^\res & H_{\ql}^{1}(\gal_{\K},\dlt)^{\Fr}\ar[r] &  0,
}
\] where the top short exact sequence is (\ref{eq:dualkottwitz}), i.e. the smooth dual  of the Kottwitz short exact sequence (\ref{eq:kottwizforK}), the bottom short exact sequence comes from the inflation-restriction sequence applied to the action of $W_K$ on the dual torus $\dlt$, the left vertical arrow is a canonical identification, the middle vertical isomorphism $\phi$ is the $\ell$-adic local Langlands correspondence for $T$ and the rightmost vertical isomorphism $\psi$ is (\ref{eq:psidef}) obtained from Theorem \ref{thm:main2} and the sheaf-function correspondence.
\end{theorem}
Note that the commutativity of the square on the right gives the compatibility of the inertial local Langlands correspondence for character sheaves on $L^+T$ (i.e. Theorem \ref{thm:main2}) with the classical correspondence due to Langlands via the sheaf-function correspondence. This result is proved in \S\ref{sec:pfmain3}.\\

\noindent{\textit{\bf Related works.}} In the case where the residue field $k$ is of positive characteristic, some closely related questions have been studied previously in \cite{Beg80} (mixed characteristic case) and \cite[\S8,10]{Suz20}. In \cite[\S8]{Suz20}, the author studies duality in the derived category of sheaves on a certain category (site) of fields and applies it to the study of tori over local fields with a perfect residue field of positive characteristic. Our approach is more elementary and classical, and also allowing $k$ to be of characteristic 0. Another difference is that we work with group of units $\ql^\times$ with the $\ell$-adic topology and not $\Q/\Z$ with the discrete topology (see also Remark \ref{rk:smoothvsladic} below). This is because we are primarily interested in studying multiplicative $\ql$-local systems on $L^+T$. We note that the spaces parametrizing multiplicative $\ql$-local systems on (connected) unipotent algebraic groups over $k$ (cf. \cite{BD06}) and on tori over $k$ (cf. \cite{GL}) behave quite differently. In the former case, it is a perfect commutative unipotent group scheme over $k$ known as the Serre dual of the connected unipotent algebraic group and leads to the notion of Fourier-Deligne transform of $\ell$-adic sheaves. On the contrary, in the case of tori defined over $k$, the space of multiplicative $\ql$-local systems is naturally a scheme defined over $\ql$ and leads to the notion of the Mellin transform as studied in \cite{GL}. In the latter case, there exists a rich supply of multiplicative $\ql$-local systems which do not have finite monodromy. In the theory of character sheaves on general algebraic groups  both the Fourier-Deligne transform as well as the Mellin transform play an important role (see also \cite{Desh:16,Desh:17}). For groups of the form $L^+T$ or $LT$, we hope to study this in future work.

The relationship between character sheaves and the local Langlands correspondence for tori over local fields $K$ has also been explored by Cunningham and Roe in \cite{CR18}. In their approach, the authors define and study (Frobenius-fixed) multiplicative local systems on possibly disconnected commutative (pro-)algebraic groups over finite fields and establish a sheaf-function correspondence for them. 
They then use the classical local Langlands correspondence for tori to obtain a description of the collection of Frobenius-fixed multiplicative local systems on the disconnected group $LT$ in terms of local Langlands parameters, which is closely related to Theorem \ref{thm:main3} above. The main difference in our approach is that we do not use the local Langlands correspondence for tori, but rather directly give a parametrization of  the multiplicative local systems on the connected group $L^+T$ and then prove that our parametrization is compatible with the classical local Langlands correspondence for tori. 

\begin{remark}\label{rk:smoothvsladic}
    In the setting of Theorems \ref{thm:main2}, \ref{thm:main3}, there often is a rich supply of multiplicative $\ql$-local systems on $L^+T$ which do not have finite monodromy. However (in the setting of Theorem \ref{thm:main3}), we will see that all the {\emph{Frobenius-fixed}} multiplicative $\ql$-local systems on $L^+T$ have finite monodromy. In particular, there is a rich supply of multiplicative $\ql$-local systems on $L^+T$ which are not fixed by any power of the Frobenius.
\end{remark}

In our approach the residue field $k$ can be any algebraically closed field. Our approach follows that of \cite{BD06,Desh:16,Desh:17} which study character sheaves on unipotent and solvable groups. In particular, we only consider multiplicative local systems on the connected commutative pro-algebraic group $L^+T$. We compute the fundamental group $\pi_1(L^+T)$ and directly give a parametrization of all multiplicative local systems on $L^+T$ in terms of inertial local Langlands parameters. We define character sheaves on a disconnected commutative algebraic group $G$ to simply be the translates of the multiplicative local systems on the neutral connected component $G^\circ$ to the other connected components $gG^\circ\subset G$. With this approach, the relationship between the Frobenius-fixed character sheaves on $G$ and the characters of (all pure inner forms of) $G(\Fq)$ is described by a ``discrete Fourier transform'' or a ``crossed S-matrix''.

Let us now describe the organization of the remainder of the text. In \S\ref{sec:prel} we recall various preliminaries needed to state and prove our results. We state the notations and conventions that are used throughout the paper in \S\ref{sec:notconv}. In \S\ref{sec:proalg} we recall the notion of pro-algebraic groups, and the definition of the fundamental group of a commutative pro-algebraic group due to Serre. In \S\ref{sec:CSonCCG}, we study multiplicative local systems on connected commutative pro-algebraic groups and their relationship with 1-dimensional characters, in case the pro-algebraic group is defined over a finite field. In \S\ref{sec:nerongreenberg}, we recall the notion of the connected N\'eron model of a torus $T$ defined over a local field and also that of the Greenberg (or positive loops) functor. In \S\ref{sec:canonclass}, we recall Serre's construction of the isomorphism (\ref{th}) in terms of the canonical class. Finally in \S\ref{sec:pfmain} we give the proofs of all our main results. In Appendix \ref{sec:ladicduality} we prove some auxiliary results related to $\ql$-characters of a certain class of abelian profinite groups which will be needed in the main part of the paper. In Appendix \ref{sec:langproalg} we recall Lang's theorem for general connected pro-algebraic groups defined over finite fields.

\section*{Acknowledgments}
We are grateful to Clifton Cunningham and Sandeep Varma for very helpful discussions. This work was supported by the Department of Atomic Energy, Government
of India, under project no.12-R\&D-TFR-5.01-0500.

\section{Preliminaries}\label{sec:prel}
\subsection{Notations and conventions}\label{sec:notconv}
     Throughout this article, $\K$ will denote a complete discrete valuation field with ring of integers $\O$ and with an algebraically closed residue field denoted by $k$. We will also consider finite Galois extensions $\E$ of $\K$. Let $\Gamma$ be the Galois group of $\E$ over $\K$. Let $\gal_{\E}$ and $\gal_{\K}$ denote the absolute Galois groups of $\E$ and $\K$ respectively. 
     We then have the following short exact sequence of groups,
     \begin{equation}\label{eqn:exactseq}
         1\to\gal_{\E}\to\gal_{\K}\to\Gamma\to 1.
     \end{equation}
    Let, \begin{equation}
    \mathscr{G}_{\E/\K}:=
        \gal_{\K}/\overline{[\gal_{\E},\gal_{\E}]}.
    \end{equation}
   Then the above exact sequence gives us the following.
     \begin{equation}\label{galois}
    1\to\gal_{\E}^{\ab}\to\mathscr{G}_{\E/\K}\to\Gamma\to 1
     \end{equation}
     
     Observe that $\gal_{\E}^{\ab}$ is a finite index normal abelian subgroup of  $\mathscr{G}_{\E/\K}$. 
      \par Let $T$ be a torus  defined over $\K$. Let $\E$ denote a finite Galois extension of $\K$ over which $T$ splits. We have the character lattice $\chg(T)=\Hom(T_{\E}, \G_{m,\E})$ and the co-character lattice $\cchg(T)=\Hom(\G_{m,\E},T_{\E})$. The group $\gal_{\K}$ acts on the lattices $\chg(T)$ and $\cchg(T)$ via the finite quotient $\gal_{\K}\twoheadrightarrow \Gamma=\gal(\E/\K)$, with the subgroup $\gal_{\E}$ acting trivially.
      \par Associated with $T$, we have the dual torus $\dt$ which is defined over $\Z$. For any commutative ring $R$, its $R$-points are given as $\dt(R)=X^*(T)\otimes R^\times$. In particular, we fix a prime number $\ell$ invertible in the residue field $k$ and consider
      \begin{equation}
      \dlt=X^\ast(T)\otimes \mql.   
      \end{equation}
      As $\Gamma$ acts on the character group $\chg(T)$, it also acts on the dual torus $\dt$. We define ${}^LT:=\dt\rtimes \gal_{\K}.$

       \noindent We now define  the abelian group $H^1_{\ell\text{-adic}}(\gal_{\K},\dlt)$. Let $Z^{1}_{\ql}(\gal_{\K},\dlt)$ be the subgroup of $Z^{1}(\gal_{\K},\dlt)$ consisting of those 1-cocycles that are continuous for $\ell$-adic topology on the target and the canonical topology on $\gal_{\K}$. The group of 1-coboundaries $B^{1}(\gal_{\K},\dlt$ is in fact a subgroup of $Z^{1}_{\ql}(\gal_{\K},\dlt)$. We define 
       \[H^1_{\ell\text{-adic}}(\gal_{\K},\dlt):=Z^{1}_{\ql}(\gal_{\K},\dlt)/B^{1}(\gal_{\K},\dlt)\leq H^{1}(\gal_{\K},\dlt).\] Now given a 1-cocycle $f:\gal_{\K}\to\dlt$, it determines a map $\widetilde{f}:\gal_{\K}\to { }^LT(\ql)$ such that (\ref{eq:illpcommute}) commutes. The cocycle condition translates to this being a homomorphism. The co-boundary maps translate to conjugation by the corresponding element of $\dlt$. As mentioned in the introduction, this induces a canonical bijection between $H^1_{\ell\text{-adic}}(\gal_{\K},\dlt)$ and $\Hom_{\gal_{\K},\ell\text{-adic}}\left.(\gal_{\K},{ }^LT(\ql))\middle/\dt(\ql)\right.$. The elements of $H^1_{\ell\text{-adic}}(\gal_{\K},\dlt)$  are known as the inertial local Langlands parameters.

       We recall the following consequence of the Baire category theorem.

\begin{lemma}\label{lemma:profinitecon}
    Let $A$ be any profinite group and $G$ an algebraic group defined over $\Q_{\ell}$. Then for any continuous homomorphism (with respect to the $\ell$-adic topology on the target) $\rho:A\to G(\ql)$, we have the $\im(\rho)\subset G(F)$ where $F$ is a finite extension of $\Q_{\ell}$.  
 \end{lemma}
 \begin{proof}
     Observe that \(\rho(A)\) is compact and hence a Baire space. We write \[\rho(A)=\bigcup_{i\in I}\rho(A)\cap G(F_i)\] where \(F_i\)'s are finite extensions of \(\Q_{\ell}\), each intersection \(\rho(A)\cap G(F_i)\) is closed and \(I\) is countable. From the Baire property, it follows that there exists \(i \in I\) such that \(\rho(A)\cap G(F_i)\) contains an open subset of \(\rho(A)\). This further implies that the index of the subgroup \(\rho(A)\cap G(F_i)\) in \(\rho(A)\) is finite. Thus, \(\rho(A)\subseteq G(F)\) where \(F\) is a finite extension of \(\Q_{\ell}\) containing \(F_i\) such that \(G(F)\) contains all the coset representatives of \(\rho(A)\cap G(F_i)\) in \(\rho(A)\).
 \end{proof}
      
       \par We use $K$ to denote a non-Archimedean local field with finite residue field $\Fq$ with $q=p^r$, where $p\neq\ell$ is a prime. Let $\Ocal$ denote the ring of integers of $K$.
       In this setting, we will use $\K$ to denote the completion of the maximal unramified extension of $K$. Thus in this case, we have the following short exact sequence, 
       \begin{equation}
           1\to\gal_{\K}\to\gal_K\to \gal(\Fqcl/\Fq)\to 1.
       \end{equation}
       The $\Fq$-Frobenius automorphism is a topological generator of $\gal(\Fqcl/\Fq)=\widehat{\Z}$. We lift it to an element of $\gal_K$. This element determines an automorphism of $\K$. Any lift of the Frobenius determines the same automorphism of $\K$. We will denote it by $\Fr_{K}$, or often simply as $\Fr:\K\to \K$. Hence we have the unique Frobenius automorphism $\Fr:\K\to\K$.
       \par From the above short exact sequence  and the inclusion  $\Z\subset\gal(\Fqcl/\Fq)$ we obtain the following short exact sequence of topological groups
       \begin{equation}
       1\to\gal_{\K}\to W_K\to \Z \to 1    
       \end{equation}
       where $W_K$ denotes the Weil group for the field $K$.
       \begin{remark}
      From the above, we get the exact sequence
      \begin{equation}
      0\to \gal_{\K}/\overline{[W_K,W_K]}\to W_K^{\ab}\to \Z\to 0.
  \end{equation} The group $\gal_{\K}/\overline{[W_K,W_K]}$ can be naturally identified with ${(\gal^{\ab}_{\K})}_{\Fr}$. 
   Moreover, local class field theory gives a canonical identification of the groups ${(\gal_{\K}^{\ab})}_{\Fr}$ and $\Ocal^{\times}$ that fit in the commutative diagram
   \begin{equation}
       \begin{tikzcd}[column sep=0pt]
           \ \ \Ocal^{\times}\arrow[d,"\re","\cong"']   &\subset & K^{\times}\arrow[d,"\re","\cong"']\\
           (\gal_{\K}^{\ab})_{\Fr} & \subset & W_K^{\ab}. 
       \end{tikzcd}
   \end{equation}
      \end{remark}
       Let $E$  denote a finite Galois extension of $K$. We let $\widetilde{\Gamma}$ denote the Galois group of $E$ over $K$. Let $W_E$ and $W_K$ be the Weil groups for $E$ and $K$ respectively. Then we have the following short exact sequences
     \begin{equation}
         1\to W_E \to W_K \to \widetilde{\Gamma} \to 1
     \end{equation}
    Using a similar notation as before, let,
    \begin{equation}
    W_{E/K}:=
        W_K/\overline{[W_E,W_E]},
    \end{equation}
    we have the short exact sequence
    \begin{equation}\label{weil}
         1\to W_{E}^{\ab} \to W_{E/K} \to \widetilde{\Gamma} \to 1.
     \end{equation}
 
     \par For a finite group $\Gamma$, a $\Gamma$-module $M$ and $i\in\Z$, we denote by $\widehat{H}^{i}(\Gamma, M)$ the $i^{th}$ Tate cohomology group. We refer to  Chapter 8 in \cite{Serre1967} for more. 
    
     \subsection{Commutative Pro-algebraic Groups}\label{sec:proalg}
     In this paper we will consider certain group schemes which occur as the limit of inverse systems of commutative algebraic groups. In \cite{Serre1961} the author describes the group of units of the discrete valuation ring $\O^\times$ as an inverse limit of commutative algebraic groups defined over the residue field $k$. Such objects are called commutative pro-algebraic groups. We recall the definition of the fundamental group of such objects from \cite{Serre1960}. This notion was used by Serre to give a geometric analogue of local class field theory.
     \par Since we will work in the setting of perfect group schemes over $k$, let us recall the definition of a perfect scheme. A scheme over a field of characteristic $p\neq 0$ is called perfect if the map $f\mapsto f^p$ on the local sections of the structure sheaf is an isomorphism of sheaves. We call any scheme over a characteristic $0$ field a perfect scheme. Perfect schemes over the algebraically closed field $k$  form a sub-category of the category of all schemes over $k$. The inclusion functor from the sub-category of perfect schemes to the category of all schemes admits a right adjoint functor called the perfectization functor \cite{Gr1965},\cite[\S A.3]{Boy2013}. Passing from a scheme to its perfectization does not change the Zarisiki or the \'etale topology or the $k$-points of the underlying scheme.  
     \begin{definition}
         A quasi-algebraic group over $k$ is a perfect group scheme over $k$ such that it is isomorphic to perfectization of a finite type group scheme over $k$.
     \end{definition}
     
      \noindent We now recall the definition of commutative pro-algebraic groups from \cite[\S 2.1]{Serre1960}:
     
     \begin{definition}
	A commutative pro-algebraic group $G$ defined over $k$ is a commutative group scheme over $k$ together with a non-empty collection of subgroup schemes $S$ such that for any $H \in S$ the quotient $G/H$ exists and is a quasi-algebraic group over $k$, and this data satisfies the following:
	\begin{enumerate}
		\item $H, H'\in S \implies H\cap H \in S'$.
		\item If $H\in S$, the subgroups containing  $H$ in $S$ are the inverse images of the closed subgroups of $G/H$.
		\item If $H\subseteq H'$ and $H,H'\in S$, then the homomorphism $G/H\to G/H'$ is a homomorphism of quasi-algebraic groups.
		\item The natural map $G \to \varprojlim G/H $ is an isomorphism to the projective limit of $G/H, H\in S$ as commutative group schemes.
	 \end{enumerate} 
         \end{definition}
     The category $\mathscr{P}$ of commutative pro-algebraic groups with morphisms as defined in \cite[\S 2.1]{Serre1960} is an abelian category. 
     
     \par A commutative pro-algebraic group $G$ is connected if  $G/H$ is connected for each $H\in S$. Let $G^\circ$ be the connected component containing the identity element.  It is the smallest subgroup of $G$ such that $G/G^\circ$ is profinite. We also have $G/G^\circ=\varprojlim (G/H)/{(G/H)}^\circ$ \cite[\S5.1]{Serre1960}. 
     \par Denote by $\mathscr{P}_{0}$ the category of abelian profinite groups with morphisms as group homomorphisms, or in other words, it is the full subcategory of $\mathscr{P}$ of 0-dimensional commutative pro-algebraic groups.
     \par Let $\pi_0:\mathscr{P} \to \mathscr{P}_{0}$ be the functor given by $$\pi_0(G):=G/G^\circ.$$  This functor is right exact \cite[\S5.1]{Serre1960}. The left derived functors of $\pi_0$ are denoted by $\pi_i$. In particular:
     \begin{definition}\label{defn:pi1}
     The first left derived functor $\pi_1(G)$ is called the fundamental group of $G$.
     \end{definition}
      The functor $\pi_1$ commutes with inverse limits (See \cite[\S 5.3]{Serre1960}).
      It is proved in \cite[\S 10]{Serre1960} that $\pi_i(G)=0, \forall i\geq 2$ for any commutative pro-algebraic group $G$. This shows that $\pi_1$ is in fact a left exact functor. Thus a short exact sequence $0\to A\to B\to C\to 0$ in $\mathscr{P}$ 
      gives a long exact sequence
       \begin{align}\label{eq:les}
     0 \to \pi_1(A)\to \pi_1(B)\to \pi_1(C)\xrightarrow{\delta} \pi_0(A)\to \pi_0(B)\to \pi_0(C)\to 0.
    \end{align}    
      In particular if $A$ is connected, i.e. if $\pi_0(A)=0$, then we have the short exact sequence
    \begin{equation}
        0 \to \pi_1(A)\to \pi_1(B)\to \pi_1(C)\to 0
    \end{equation}
    i.e.  we have
    \begin{equation}\label{eq:pi_1 quo}
     \pi_1(B/A)=\pi_1(B)/\pi_1(A).   
    \end{equation}
    We have for a pro-algebraic group $G$, $\pi_1(G^{\circ})=\pi_1(G)$.
    \par There is a notion of universal cover of a commutative pro-algebraic group given in \cite[\S 6.2]{Serre1960}. We recall it here.
    A commutative pro-algebraic group is said to be simply connected if $\pi_1(G)=0$. For any $G\in \mathscr{P}$, $\exists$ a connected, simply connected commutative pro-algebraic group $\overline{G}$ with an exact sequence:
    \begin{equation}\label{eq:defnunicover}
        0\to\pi_1(G)\to\overline{G}\to G\to\pi_0(G)\to 0.
    \end{equation}
    The group $\overline{G}$ is called the universal cover of $G$.
    \par The construction of the universal cover for a connected quasi-algebraic group is given in \cite[\S 6.4]{Serre1960} as the inverse limit of connected groups $G_{f}$ with an isogeny $f:G_{f}\to G$. Hence the fundamental group of such objects is
    \begin{equation}\label{eq:pi_1}
      \pi_1(G)=\varprojlim \ker(f).  
    \end{equation}
    It is proved in \cite[\S 2]{Serre1960} that the universal cover functor commutes with inverse limits. Hence we obtain a description for the universal cover of any pro-algebraic group $G$.
    \begin{remark}\label{remark:et}
    For any $k$-scheme $X$, the \'etale fundamental group of $X$, denoted by $\pi_{1}^{\acute{e}t}(X)$, is the inverse limit of automorphism groups of all finite \'etale covers of $X$, whereas $\pi_1$  defined here (Definition \ref{defn:pi1}) is defined only for commutative pro-algebraic groups. It is the inverse limit of automorphism groups of finite \'etale covers which are also isogenies as above. Hence for a pro-algebraic group $G$ there is a surjective map $\pi_{1}^{et}(G)\twoheadrightarrow \pi_1(G)$.
    \end{remark}

We next consider the case $k=\Fqcl$. Let $G$ be a connected commutative pro-algebraic group over $k$ equipped with an $\Fq$-Frobenius $\Fr:G\to G$ as in Appendix \ref{sec:langproalg}. By Lang's theorem  for connected pro-algebraic groups (Theorem \ref{thm:langprofin}), for any positive integer $n$, the $n$-th Lang isogeny $\Lang_n: G \to G$ defined by
            $g\mapsto g \Fr^{n}(g^{-1})$ gives us the following short exact sequence in $\mathscr{P}$
    \begin{equation}\label{eq:Lang ses}
        0 \to G(\F_{q^n})\to G\xrightarrow{\Lang_n} G\to 0.
    \end{equation}
    We then have the following lemma:
    \begin{lemma}\label{lemma:pi1to fqpts}
       In the setting above, we have the canonical maps $\pi_1(G)\twoheadrightarrow \pi_1(G)_{\Fr^n}\xrightarrow{\cong}G(\mathbb{F}_{q^n})$, i.e. the profinite abelian group $G(\F_{q^n})$ can be canonically identified with the $\Fr^n$-coinvariants of $\pi_1(G)$.
    \end{lemma}
    \begin{proof}
        Consider the long exact sequence \eqref{eq:les} associated with the short exact sequence \eqref{eq:Lang ses}:  
        \begin{align}
        \dots\to \pi_1(G(\F_{q^n})) = 0 \to \pi_1(G) \xrightarrow[]{\pi_1(\Lang_n)}\pi_1(G) \to G(\F_{q^n})\to 0=\pi_0(G)\to\dots 
      \end{align}
      Observe that the $\Fr^n$-coinvariants of $\pi_1(G)$ are exactly $\pi_1(G)/\pi_{1}(\Lang_n)(\pi_1(G))$. This proves the lemma.
    \end{proof}

   \subsection{Character sheaves on connected commutative pro-algebraic groups}\label{sec:CSonCCG}
    \begin{definition}
        Let $G$ be a connected commutative pro-algebraic group defined over an algebraically closed field $k$ of characteristic $p\geq 0$ and let $\ell$ be a prime number invertible in $k$. A  multiplicative $\ql$-local system, or a character sheaf on $G$ is a rank one $\ql$-local system $\mathcal{L}$ on $G$ such that $\alpha^*\mathcal{L}\cong \mathcal{L}\boxtimes\mathcal{L}$, where $\alpha:G\times G\to G$ is the multiplication map. The isomorphism classes of multiplicative local systems on $G$ form an abelian group under tensor product, which we denote by $\CS(G)$.
    \end{definition}

    \begin{remark}
    Let $A$ be a commutative profinite group. We can analogously define the notion of multiplicative $A$-torsors on a connected pro-algebraic group $G$. Then the isomorphism classes of multiplicative $A$-torsors on $G$ can be naturally identified with the isomorphism classes of central extensions of $G$ by $A$
    (see \cite[\S1]{Kam2009} and \cite[\S 7.2]{Boy2010}). This further gives us a natural identification of abelian groups $$\CS(G)=\Hom_{\ell\text{-adic}}(\pi_1(G),\ql^\times)$$ as follows:
    \\
    Let $\mathcal{L}_{\chi}$ be any rank 1 local system on $G$ corresponding to a continuous ($\ell$-adic) character $\chi:\pi_1^{\acute{e}t}(G)\to \ql^\times$ of the \emph{\'etale fundamental group} of $G$. This corresponds to an $\im(\chi)$-torsor $X\onto G$, where $\im(\chi)\subseteq \ql^\times$ is a pro-finite subgroup. Moreover, the local system $\mathcal{L}_\chi$ is multiplicative if and only if the corresponding $\im(\chi)$-torsor $X\onto G$ is multiplicative, which is in turn equivalent to $X\onto G$ having the structure of a central extension of $G$ by $\im(\chi)$. In other words, $\mathcal{L}_\chi$ is a multiplicative local system if and only if the character $\chi$ factors through the quotient $\pi_1^{\acute{e}t}(G)\onto\pi_1(G)$.

    Also note that by Lemma \ref{lemma:profinitecon}, $\im(\chi)\subset \Ocal_{F}^\times$ for some finite extension $F$ of $\mathbb{Q}_{\ell}$.
    \end{remark}

     Lusztig defined character sheaves on connected commutative algebraic groups $G$ defined over a finite field $\Fq$ in \cite[\S 5]{Lus06} and related them to the character theory of the finite abelian groups $G(\Fq)$. We will recall these results below following \cite[\S 1.5]{BD06}.

    Let $G$ be a connected commutative algebraic group over $\Fqcl$ equipped with an $\Fq$-structure given by a Frobenius map $\Fr:G\to G$. For each positive integer $n$, a $\Fr^{n}$-invariant $\ql$-multiplicative local system  is a multiplicative local system  with an isomorphism $\psi_{n,\mathcal{L}}:{(\Fr^{n})}^{\ast}(\mathcal{L})\cong \mathcal{L}$, which is chosen such that it is identity on the stalk of $\mathcal{L}$ at $1\in G$. Trace functions for such local systems  $\mathcal{L}$ are  defined in \cite{Del77}, \cite[\S 5]{Lus06} and \cite[\S 1.5]{BD06} as follows:
    \begin{equation}\label{eq:defn trace fn}
    \begin{split}
      \Tr_{n,\mathcal{L}}:  G(\mathbb{F}_{q^n})\to \mql\\
      \Tr_{n,\mathcal{L}}(g):= \tr(\mathcal{L}_g=\mathcal{L}_{\Fr^n(g)}\xrightarrow[\cong]{\psi_{n,\mathcal{L},g}} \mathcal{L}_g),
     \end{split}
    \end{equation}
    where $\mathcal{L}_g$ denotes the stalk of $\mathcal{L}$ at $g\in G(\mathbb{F}_{q^n})$. It was proved by Lusztig that these trace of Frobenius functions provide a bijection between the set of $\Fr^{n}$-invariant multiplicative $\ql$-local systems on $G$ and 1-dimensional characters of $G(\mathbb{F}_{q^n})$ for any $n\in\Z_{>0}$.
    To construct the map in the other direction, let \begin{equation} \label{eq:defn Psi_n}
    \Psi_n:\Hom(G(\mathbb{F}_{q^n}),\mql)\xrightarrow{\Psi_n}\Hom_{\ell\text{-adic}}(\pi_{1}(G),\mql)    
    \end{equation} be the map induced by pulling back a character along  the  homomorphism $\pi_1(G)\twoheadrightarrow G(\mathbb{F}_{q^n})$ given by Lemma \ref{lemma:pi1to fqpts}. Since ${\pi_1(G)}_{\Fr^n}=G(\mathbb{F}_{q^n})$, we have $\im(\Psi_n)\subset \Hom_{\ell\text{-adic}}(\pi_{1}(G),\mql)^{\Fr^n}$. 
    We recall Lusztig's result following \cite[\S1.5, \S 1.8]{BD06}:
    \begin{theorem}\label{th: quasi alg sh-fn}
    Let $G$ be a connected, commutative algebraic group  over $\Fqcl$ equipped with an $\Fq$-structure. 
    \begin{itemize}
        \item[(i)] The morphism of abelian groups $\Psi_n$  \eqref{eq:defn Psi_n} gives us an isomorphism
        \begin{equation}\label{sh-fn}
            \Hom(G(\mathbb{F}_{q^n}),\mql)\xrightarrow{\Psi_n}\Hom_{\ell\text{-adic}}(\pi_{1}(G),\mql)^{\Fr^n}=\CS(G)^{\Fr^n}
    \end{equation}
    and for every character $\chi:G(\mathbb{F}_{q^n})\to \mql$ 
     $$\Tr_{n, \Psi_n(\chi)}=\chi.$$
     \item[(ii)] For $n\in\N$, we have the norm homomorphism  $\nm_n:G(\mathbb{F}_{q^n})\to G(\Fq)$  defined as,
     \begin{equation}\label{norm map}
      \nm_n=\prod_{s=1}^{n} \Fr^{s}.
    \end{equation}
    The homomorphism on the Pontryagin duals induced by the norm map defines an isomorphism 
    \begin{equation}\label{compatibilty}  \widehat{\nm}_n:\Hom(G(\Fq),\mql)\xrightarrow{\cong}\Hom(G(\mathbb{F}_{q^n}),\mql)^{\Fr}.
    \end{equation}
   \item[(iii)] The following diagram is commutative
    \begin{center}
     \begin{tikzcd}
\Hom(G(\Fq),\mql)\arrow[r,"\widehat{\nm}_n","\cong"']\arrow[d,"\Psi_1","\cong"']& \Hom(G(\mathbb{F}_{q^n}),\mql)^{\Fr}\arrow[r,hook]\arrow[d,"\Psi_n","\cong"'] & \Hom(G(\mathbb{F}_{q^n}),\mql)\arrow[d,"\Psi_n","\cong"'] \\
\Hom_{\ell\text{-adic}}(\pi_{1}(G),\mql)^{\Fr}\ar[equal]{r}& \Hom_{\ell\text{-adic}}(\pi_{1}(G),\mql)^{\Fr}\arrow[r,hook] & \Hom_{\ell\text{-adic}}(\pi_{1}(G),\mql)^{\Fr^n}. 
    \end{tikzcd}
    \end{center}
   \end{itemize}
    \end{theorem}
    The above theorem also holds if we consider quasi-algebraic connected commutative groups as in \cite{BD06} since passing to the perfectizations does not affect the arguments.
     We will now establish a similar correspondence for  pro-algebraic groups of certain type. We first prove a simple lemma: 
     \begin{lemma}\label{lemma:pro-quo}
     Let $P$ be a profinite group with $P=\varprojlim P_i$, where $\{P_i\}_i$ form a directed inverse system  of profinite groups. Then any continuous homomorphism from $P$ to a finite group factors through some $P_i$.
      \end{lemma}
   \begin{proof}
   Let $f:P\to A$ be a continuous homomorphism to a finite  group $A$. Let $N_i$ denote the the kernel of the map $P\to P_i$ and $N$ denote the kernel of $f$. Each $N_i$, $N$ is a closed subset of $P$ and  $\cap_i N_i =\{1\}$. Taking the complements, $\cup_i {N_i}^{\mathsf{c}}= P\setminus\{1\}\supseteq {N}^{\mathsf{c}}$. Since $N$ has finite index in $P$, ${N}^{\mathsf{c}}$ is compact. Hence it is covered by finitely many ${N_i}^{\mathsf{c}}$. Since the indexing set is directed, there exists an $i_0$ such that $N_{i_0}\subset N$. Hence $f$ factors through  $P_{i_0}$.
    \end{proof}     
     \noindent We now make the following assumptions on the pro-algebraic group $G$ defined over an algebraically closed field $k$:
     \begin{assumption}\label{assump1}
     Assume that $G$ is a connected, commutative pro-algebraic group defined over an algebraically closed field $k$, that has a filtration $G\supset U^{(0)}\supset U^{(1)}\supset U^{(2)}\supset U^{(3)}\supset\dots$ by connected pro-unipotent subgroups $U^{(i)}$, such that each $G/U^{(i)}$ is a commutative quasi-algebraic group and that $G=\varprojlim G/U^{(i)}$, and hence $U^{(0)}=\varprojlim U^{(0)}/U^{(i)}$.
      \end{assumption}
      Since the functor $\pi_1$ commutes with inverse limits and since $U^{(i)}$ are connected, we have using \eqref{eq:pi_1 quo}
\begin{equation}\label{eq:pi1,inv limit}
     \begin{split}
      \pi_1(G)= & \varprojlim\pi_1(G/U^{(i)})=\varprojlim\pi_1(G)/\pi_1(U^{(i)})\\  
      \pi_1(U^{(0)})= & \varprojlim\pi_1(U^{(0)}/U^{(i)})=\varprojlim\pi_1(U^{(0)})/\pi_1(U^{(i)}).
     \end{split}
    \end{equation}
    \begin{lemma}\label{lemma:pi1 fin-im}
         Let $G$ be a connected commutative pro-algebraic group satisfying Assumption \ref{assump1}. Then any $\chi\in\Hom_{\ell\text{-adic}}(\pi_1(G),\mql)$ factors through $\pi_1(G/U^{(i)})$ for some $i$. In other words, any multiplicative local system on $G$ is a pull-back of a multiplicative local system on $G/U^{(i)}$ under the natural projection $G\to G/U^{(i)}$. If $k$ is of characteristic 0, then in fact $\pi_1(G)=\pi_1(G/U^{(0)})$ and hence any multiplicative local system on $G$ is a pull-back of a multiplicative local system on $G/U^{(0)}$.
     \end{lemma}
     \begin{proof}
     If characteristic of $k$ is $0$, then $\pi_1(U^{(0)})$ is trivial. Hence the last statement above is clear.
     
     We now assume that the characteristic of $k$ is a prime number $p$. Since $\chi$ is a continuous (\(\ell\)-adic) homomorphism, $\im(\chi)\subset F^\times$ where $F$ is a finite extension of $\Q_{\ell}$ by Lemma \ref{lemma:profinitecon}. Additionally  $\pi_1(G)$ is a profinite group, hence compact and hence $\im(\chi)\subset \Ocal_{F}^{\times}$ where $\Ocal_{F}$ is the ring of integers of $F$.  Consider $\chi|_{\pi_1(U^{(0)})}$. Since $U^{(0)}$ is a pro-unipotent group, $\pi_1(U^{(0)})$ is a pro-$p$ group \cite[\S 8]{Serre1960}. Hence the $\ell$-primary part $(\pi_1(U^{(0)}))_{\ell}$ is trivial. The structure of the group of units $\Ocal_{F}^{\times}$ is of the form $$\Ocal_{F}^{\times} \cong \Z/((\ell-1)\Z)\oplus\Z/{\ell}^a\Z\oplus\Z_{\ell}^{d}$$ for some non-negative integers $a$ and $d$ (Chapter 2,\cite[\S 5]{Neu1992}). In particular, the $p$-primary part of $\Ocal_{F}^{\times}$ i.e. $(\Ocal_{F}^{\times})_p$ is finite. Since $\chi$ preserves the $p$-primary part, the image $\chi({\pi_1(U^{(0)})})$ is finite. Using \eqref{eq:pi1,inv limit} and Lemma \ref{lemma:pro-quo} we see that $\chi|_{\pi_1(U^{(0)})}$ factors through $\pi_1(U^{(0)}/U^{(i)})=\pi_1(U^{(0)})/\pi_1(U^{(i)})$ for some $i$, i.e. $\pi_1(U^{(i)})\subset\ker\chi$. Hence $\chi$ factors through $\pi_1(G/U^{(i)})=\pi_1(G)/\pi_1 (U^{(i)})$ as desired.
        \end{proof}

    \begin{definition}
    Suppose that $G$ satisfies Assumption \ref{assump1}. Depth of a character sheaf $\mathcal{L}$, corresponding to a $\ql$-character $\chi:\pi_1(G)\to\mql$, is the minimal non-negative integer $i$ such that $\chi$ factors through $\pi_1(G/U^{(i)})=\pi_1(G)/\pi_1(U^{(i)})$. Equivalently, it is the minimal non-negative integer $i$ such that $\mathcal{L}$ is the pull-back of a character sheaf on the quasi-algebraic group $G/U^{(i)}$ under the natural projection $G\onto G/U^{(i)}$.
    \end{definition}
    Note that the depth of the character sheaf depends on the chosen filtration.
    Denote by $\CS_{\leq d}(G)\subset \CS(G)$ the subgroup of multiplicative local systems of depth less than or equal to $d\in \Z_{\geq 0}$. By Lemma \ref{lemma:pi1 fin-im}
   \begin{equation}\label{eq:CS union}
     \CS(G)=\bigcup_{d\geq 0} \CS_{\leq d}(G)=\bigcup_{d\geq 0}\CS(G/U^{(d)}).  
   \end{equation}
    \begin{assumption}\label{assump2}
    Assume that $k=\Fqcl$. In addition to Assumption \ref{assump1}, suppose that $G$ has an $\Fq$-structure given by a Frobenius $\Fr:G\to G$ such that each $U^{(i)}$ is $\Fr$-stable. Since each $U^{(i)}$ is connected, using Lang's Theorem we see that
    \begin{equation}\label{eq:G fqpoints quo}
     \begin{split}
         U^{(0)}(\Fq)=\varprojlim (U^{(0)}/U^{(i)})(\Fq)=\varprojlim U^{(0)}(\Fq)/U^{(i)}(\Fq),\\
         G(\Fq)=\varprojlim(G/U^{(i)})(\Fq)=\varprojlim G(\Fq)/U^{(i)}(\Fq).
     \end{split}
     \end{equation}
     
    \end{assumption}
    \begin{lemma}\label{lemma:fqpts fin im}
        For a connected commutative pro-algebraic group $G$ satisfying Assumptions \ref{assump1} and \ref{assump2}, any $\ql$-character $\rho:G(\Fq)\to\mql$ factors through $(G/U^{(i)})(\Fq)=G(\Fq)/U^{(i)}(\Fq)$ for some $i$. In particular $\im(\rho)$ is finite. 
    \end{lemma}
    \begin{proof}
     The groups $U^{(i)}$ are pro-unipotent. Hence $U^{(i)}(\Fq)$ are pro-$p$ groups. Now, the lemma can proved using \eqref{eq:G fqpoints quo} and Lemma \ref{lemma:pro-quo} and an argument similar to that used to prove Lemma \ref{lemma:pi1 fin-im}. 
     \end{proof}
     \begin{definition}
         Let $G$ satisfy Assumptions \ref{assump1} and \ref{assump2}. Depth of a $\ql$-character $\rho:G(\Fq)\to\mql$ is the minimum non-negative integer $i$ such that $\rho$ factors through the finite quotient $(G/U^{(i)})(\Fq)$.
     \end{definition}
     \begin{remark}\label{remark:smooth char}
     An abstract group homomorphism from a profinite group to $\mql$ is said to be smooth if its kernel contains an open subgroup, or equivalently if is continuous for the discrete topology on $\mql$. Let $\Hom_{\sm}(G(\Fq),\mql)$ denote the set of smooth homomorphisms from $G(\Fq)$ to $\mql$. In the case when $G$ is a pro-algebraic group satisfying Assumptions \ref{assump1} and  \ref{assump2}, we have  $$\Hom_{\sm}(G(\Fq),\mql)=\Hom_{\ell\text{-adic}}(G(\Fq),\mql).$$ Indeed, it is clear that $\Hom_{\sm}(G(\Fq),\mql)\subset\Hom_{\ell\text{-adic}}(G(\Fq),\mql)$ and the reverse inclusion follows as each $\ql$-character of $G(\Fq)$ has a finite image. 
     \end{remark}

     We then have the following version of Lusztig's theorem for pro-algebraic groups:
     \begin{theorem}\label{thm:sh-fn}
         Let $G$ be a connected commutative pro-algebraic group satisfying Assumptions \ref{assump1} and  \ref{assump2}. Then:
         \begin{itemize}
             \item[(i)] For each $n\in \Z_{>0}$, the homomorphism $\Psi_n$ of abelian groups obtained by pulling back $\ql$-characters via the projection $\pi_1(G)\onto G(\F_{q^n})$ from Lemma \ref{lemma:pi1to fqpts} defines an isomorphism of abelian groups $$\Hom_{\sm}(G(\mathbb{F}_{q^n}),\mql)=\Hom_{\ell\text{-adic}}(G(\mathbb{F}_{q^n}),\mql)\xrightarrow[\Psi_n]{\cong} \Hom_{\ell\text{-adic}}(\pi_1(G),\mql)^{\Fr^n}=\CS(G)^{\Fr^n}$$ such that for each $\chi\in \Hom_{\ell\text{-adic}}(G(\mathbb{F}_{q^n}),\mql)$
             $$\Tr_{n, \Psi_n(\chi)}=\chi.$$ 
             \item[(ii)]The above isomorphism preserves depth, i.e. it takes a 
             $\ql$-character of $G(\F_{q^n})$ of depth $d$ to a $\Fr^n$-invariant character sheaf on $G$ of depth $d$.
             \item[(iii)] The norm map $\nm_n$ defined in   \eqref{norm map} induces an isomorphism  $$\widehat{\nm}_n:\Hom_{\ell\text{-adic}}(G(\mathbb{F}_{q}),\mql)\to \Hom_{\ell\text{-adic}}(G(\mathbb{F}_{q^n}),\mql)^{\Fr}$$ and the following diagram is commutative
              \begin{center}
            \begin{tikzcd}
     \Hom_{\ell\text{-adic}}(G(\Fq),\mql)\arrow[r,"\widehat{\nm}_n","\cong"']\arrow[d,"\Psi_1","\cong"'] & \Hom_{\ell\text{-adic}}(G(\mathbb{F}_{q^n}),\mql)^{\Fr}\arrow[r,hook]\arrow[d,"\Psi_n","\cong"'] & \Hom_{\ell\text{-adic}}(G(\mathbb{F}_{q^n}),\mql)\arrow[d,"\Psi_n","\cong"']\\
      \Hom_{\ell\text{-adic}}(\pi_1(G),\mql)^{\Fr}\ar[equal]{r} & \Hom_{\ell\text{-adic}}(\pi_1(G),\mql)^{\Fr}\arrow[r,hook] & \Hom_{\ell\text{-adic}}(\pi_1(G),\mql)^{\Fr^n}.
           \end{tikzcd}
           \end{center}
        \end{itemize}
    
     \end{theorem}
     \begin{proof}
     By Remark \ref{remark:smooth char}, we have $\Hom_{\sm}(G(\mathbb{F}_{q^n}),\mql)=\Hom_{\ell\text{-adic}}(G(\mathbb{F}_{q^n}),\mql)$. For each $d\in \Z_{\geq 0}$, $G/U^{(d)}$ is a quasi-algebraic group with an $\Fq$-structure, hence by Theorem \ref{th: quasi alg sh-fn} (along with Lang's Theorem) we have 
     \[\Hom(G(\F_{q^n})/U^{(d)}(\F_{q^n}),\mql)\cong \CS(G/U^{(d)})^{\Fr^n}.\]
     By Lemma \ref{lemma:fqpts fin im} (resp. \ref{lemma:pi1 fin-im}) all the $\ql$-characters of $G(\F_{q^n})$ (resp. character sheaves on $G$) have finite depth. Moreover, since the projection $G\to G/U^{(d)}$ is compatible with the $\Fq$-structures we have the  commutative diagram
     \begin{equation}
     \xymatrix{
     \pi_1(G)\ar@{->>}[r] \ar@{->>}[d]& G(\F_{q^n})\ar@{->>}[r]^{\nm_n}\ar@{->>}[d]& G(\Fq)\ar@{->>}[d]\\
     \pi_1(G/U^{(d)})\ar@{->>}[r] & (G/U^{(d)})(\F_{q^n})\ar@{->>}[r]^{\nm_n}& (G/U^{(d)})(\Fq).
     }\end{equation}
     The theorem now follows from Theorem \ref{th: quasi alg sh-fn}.
      \end{proof}
      \begin{remark}
     A version of Theorem \ref{thm:sh-fn} can be stated for all connected commutative pro-algebraic groups with an $\Fq$-structure, with its proof on similar lines as given in \cite{BD06} for quasi-algebraic groups. Since the pro-algebraic groups we deal with in this paper satisfy the additional assumptions in the theorem above, we only state it under these assumptions.
     \end{remark}   
     \subsection{N\'eron Models and Greenberg Functor}\label{sec:nerongreenberg}
     Given a torus $T$ over $\K$ we wish to get a smooth model of $T$ over $\O$. We consider the connected N\'eron model of $T$ for this purpose. It is the connected component containing the identity element of the lft-N\'eron model. We denote the connected N\'eron model (which is a connected smooth group scheme over $\O$) by $\T$.  The connected N\'eron model $\T$ is functorial in tori $T$ defined over $\K$ (see \cite[\S B.7]{KP23}). For a split torus $T$ over $\K$, the co-ordinate ring of the connected N\'eron model $\T$ is given by
    \begin{equation}
        \O[\T]=\O[\cchg(T)].
    \end{equation} 
    In particular for the multiplicative group $\G_m$ over $\K$, the connected N\'eron model is $\G_{m,\O}$ the multiplicative group over $\O$. For an induced torus $T=\res_{\E/\K}\G_m$, where $\E$ is a finite separable extension of $\K$, the connected N\'eron model is given by $$\T=\res_{\mathcal{O}_{\E}/\O}\G_{m,\mathcal{O}_{\E}},$$ where $\mathcal{O}_{\E}$ denotes the ring of the integers of $\E$. 
    \par Let $\mathfrak{m}=\langle \varpi \rangle$ be the maximal ideal of $\O$ and  $k$ the algebraically closed residue field.  In \cite{Serre1961}, the group of units $\O^\times$ is given a structure of commutative  pro-algebraic group. This is done by giving ${(\O/\mathfrak{m}^{n})}^\times$ the structure of a quasi-algebraic group over $k$ and then using $\O^{\times}\cong \varprojlim{(\O/\mathfrak{m}^{n})}^\times$. 
     This construction can be generalized for any torus using the Greenberg functor applied to its connected N\'eron model. We recall the definition here. If $k$ has positive characteristic and $R$ is a perfect $k$-algebra, let $W(R)$ denote the ring of unramified Witt vectors with values in $R$. 
     
     Define the following functors: For any perfect $k$-algebra $R$
     \begin{equation}
         \mathbb{W}(R):=
         \begin{cases}
         R[[\varpi]] & \text{if $\ch(k)=\ch(\O)$} \\
         W(R)\otimes_{W(k)}\O & \text{otherwise} 
         \end{cases}
     \end{equation}
     \begin{equation}
         \mathbb{W}\left[\frac{1}{\varpi}\right](R):=
         \begin{cases}
         R((\varpi)) & \text{if $\ch(k)=\ch(\O)$} \\
         (W(R)\otimes_{W(k)}\O)\left[\frac{1}{\varpi}\right] & \text{otherwise} 
         \end{cases}
     \end{equation}
    Now for a quasi finite scheme $X$ over $\O$, a scheme $Y$ over $\K$ and a perfect $k$-algebra $R$ define
     \begin{equation}
        \lp X(R)=X(\mathbb{W}(R))
     \end{equation}
     \begin{equation}
     \lpf Y(R)=Y(\mathbb{W}\left[\frac{1}{\varpi}\right](R)).
     \end{equation}
     Observe that $\lp X(k)=X(\O)$ and $\lpf Y(k)=Y(\K)$. The functor $\lp{}$ is called the Greenberg functor and is defined in \cite{Gr1961},\cite{BG2018}. For a smooth group scheme $G$ over $\O$, $L^+G$ is a pro-algebraic group over $k$. If $G$ is also commutative, then $L^+G$ is a commutative pro-algebraic group.
     \par Throughout we will use $\lp{T}$ to denote the Greenberg functor applied to the connected N\'eron model of the torus $T$ over $\K$. It is a connected commutative pro-algebraic group over $k$. We use $\elp{T}$ to denote the Greenberg functor applied to the connected N\'eron model of the torus $T$ defined over a finite separable extension $\E$ of $\K$. For a torus $T$ over $\K$, $\elp{T}$ would mean the Greenberg functor applied to the connected N\'eron model of the torus $T_{\E}$ over $\E$ obtained by extension of scalars.

     For the connected N\'eron model, $\ \T(\O)=T(\K)^{0}\subset T(\K)$ where $T(\K)^{0}$ is the Iwahori subgroup. By \cite[\S 5]{JKYu[2015]} we have a filtration of $T(\K)$ by connected schematic subgroups $$T(\K)\supset T(\K)^{0}\supset U_{0} \supset  U_{1} \dots$$ where the subgroups $U_{n}:=T(\K)_{n+1}^{mc}$ for $n\geq 0$ are defined in \cite[\S 5]{JKYu[2015]}. Since $U_{n}$ is connected schematic we have smooth connected $\O$- group schemes $\T_n$ with $\T_n(\O)=U_n$. Hence we have the filtration of connected commutative pro-algebraic groups $L^+T\supset L^+\T_0\supset L^+\T_1\supset\cdots$. It is proved in \cite[\S5.2]{JKYu[2015]} that each $L^+\T_{n}/L^+\T_{n+1}$ is unipotent and that $\lp T=\varprojlim L^+T/\lp \T_n$. Hence $\lp{T}$ satisfies the Assumption \ref{assump1}. In case $T$ is defined over a local field $K$ (with finite residue field), then the filtration is $\Fr$-stable and hence Assumption \ref{assump2} is also satisfied in this case.

     \subsection{Serre's local class field theory and the canonical class}\label{sec:canonclass}
     Consider the reciprocity map $\re: K^{\times} \to W_{K}^{\ab}$ for a local field $K$ with finite residue field. We first describe the inverse of this map as defined in \cite[Ch XI,\S3]{Serre1967}. If $E/K$ is a finite Galois extension, then we have the isomorphisms (see {\it{loc. cit.}})
     \begin{equation}
     \begin{split}
     \gal(E/K)^{\ab}\cong \tH^{-2}(\gal(E/K), \Z)\\
     K^{\times}/\nm E^{\times}\cong \tH^{0}(\gal(E/K),E^{\times}). 
     \end{split} 
     \end{equation}
     Here $\nm: E^\times\to K^\times$ is the norm map defined as $\nm(x):=\prod_{\sigma\in \gal(E/K)} \sigma(x)$. Then the ``inverse of the reciprocity map'', namely $\gal(E/K)^{\ab}\to K^{\times}/\nm E^{\times}$, is given by taking cup product with the fundamental class in $\tH^{2}(\gal(E/K), E^{\times})$. 
     The reciprocity map $\re_E:E^{\times}\to W_{E}^{\ab}$  gives us a map of the cohomology groups $$\tH^{2}(\gal(E/K), E^{\times})\to\tH^{2}(\gal(E/K), W_{E}^{\ab}).$$ Under this map the image of the fundamental class corresponds to the short exact sequence \eqref{weil}, i.e.
     \begin{equation}
         1\to W_{E}^{\ab}\to W_{E/K}\to\gal(E/K)\to 1.
     \end{equation}
     
     Serre in \cite[\S 2.5]{Serre1961} gives a similar construction of the inverse map $\theta: \pi_1(\lp{\G_m}) \to \gal_{\K}^{\ab}$ to \eqref{th}. We recall the construction here.  As before, $\E$ will denote a finite Galois extension of $\K$ and $\Gamma$ will denote the Galois group. The ring of integers of $\E$ will be denoted by $\Ocal_{\E}$. 
      
      The natural action of $\Gamma$ on $\Ocal_{\E}^\times$ defines a $\Gamma$-action on the pro-algebraic group $\elp{\mathbb{G}_m}$, 
 whose $\Gamma$-fixed points are $\lp\G_m$. We also get a $\Gamma$-action on its universal cover  $\overline{\elp{\mathbb{G}_m}}$. Then \eqref{eq:defnunicover} gives us an exact sequence of $\Gamma$-modules:
    \begin{equation}
	     0\to \pi_1(\elp{\mathbb{G}_m})\to \overline{\elp{\mathbb{G}_m}}(k)\to \elp{\mathbb{G}_m}(k)\to 0. 
	     \end{equation}

	For any subgroup $H\subseteq \Gamma$ the above sequence gives the following long exact sequence for Tate cohomology:
	     \begin{small}
	     \begin{equation*}
	     \begin{split}
	          \to \tH^{q+1}(H,\overline{\elp{\mathbb{G}_m}}(k))\to \tH^{q+1}(H,\elp{\mathbb{G}_m}(k))\to   
              \tH^{q+2}(H,\pi_1(\elp{\mathbb{G}_m}))\to\tH^{q+2}(H,\overline{\elp{\mathbb{G}_m}}(k))\to
	     \end{split}
	     \end{equation*}
	     \end{small}
	     Since $\overline{\elp{\mathbb{G}_m}}(k)$ is cohomologically trivial \cite[\S 2.5]{Serre1961}, for all $q\in \Z$ we have an isomorphism:
	     \begin{equation} \label{iso1}
	      \tH^{q+1}(H,\elp{\mathbb{G}_m}(k))\xto{\cong} \tH^{q+2}(H,\pi_1(\elp{\mathbb{G}_m})).
	     \end{equation}
	     \\ We also have the short exact sequence of $\Gamma$-modules:
	     \begin{equation*}
         0\to \elp\G_m(k)=\Ocal_{\E}^{\times} \to \E^\times \to \Z\to 0.
	     \end{equation*}
	     This gives a long exact sequence:
       \begin{equation*}
     \dots\to \tH^q(H,\Ocal_{\E}^{\times})\to \tH^q(H,\E^\times)\to \tH^q(H,\Z)\to \tH^{q+1}(H, \Ocal_{\E}^{\times})\to \tH^{q+1}(H, \E^\times)\to \dots
	  \end{equation*}
	 Since $\E^\times$ is also cohomologically trivial \cite[\S2.2 Prop. 1]{Serre1961}, we get an isomorphism:
	  \begin{equation} \label{iso2}
	     \tH^q(H,\Z)\xto{\cong} \tH^{q+1}(H,\elp{\mathbb{G}_m}(k)). 
	  \end{equation}
	  \\From \eqref{iso1} and \eqref{iso2} for all $q\in \Z$ and for all subgroups $H$ of $\Gamma$ we have the isomorphism:
	  \begin{equation}
	      \tH^q(H,\Z)\xrightarrow[\cong]{\delta_H} \tH^{q+2}(H,\pi_1(\elp{\mathbb{G}_m})).
	  \end{equation}
        Using $q=-1$ and $q=0$ for any subgroup $H\subset\Gamma$,
	  \begin{equation}\label{eq:canonicalclass}
       \begin{split}
        \tH^1(H, \pi_1(\elp{\mathbb{G}_m}))={1}\\
        \tH^2(H,\pi_1(\elp{\mathbb{G}_m}))\cong \Z/|H|\Z.
        \end{split}
	  \end{equation}Thus $\tH^2(H,\pi_1(\elp{\mathbb{G}_m}))$ is a cyclic group. The element $u_{\E/\K}:=\delta_{\Gamma}(\overline 1)$ generates $\tH^2(\Gamma,\pi_1(\elp{\mathbb{G}_m}))$. For any subgroup $H$, its restriction to $H$, $\res(u_{\E/\K}) \in \tH^2(H,\pi_1(\elp{\mathbb{G}_m}))$ is also a generator. 
     The element $u_{\E/\K}$ is called the canonical class. 
       Cup product with the canonical class gives an isomorphism: 
       \begin{equation}\label{eq:defn-delta}
      \delta_{\E/\K}:\Gamma^{\ab}=\tH^{-2}(\Gamma,\Z)\xrightarrow[\cong]{\_\cup u_{\E/\K}}\tH^0(\Gamma,\pi_1(\elp \G_m))=\pi_1(\lp \G_m)/\nm(\pi_1(\elp \G_m)).
       \end{equation}
     Here $\nm$ is the norm map $\pi_1(\lp{\G_{m,\E}}) \to \pi_1(\lp{\G_m})$ for the $\Gamma$-action.
     The induced map $\gal_{\K}^{\ab}\to\pi_1(\lp\G_m)$ at the level of inverse limits   is the inverse of the map $\theta:\pi_1(\lp\G_m)\to \gal_{\K}^{\ab}$ (see \cite[2.5]{Serre1961}).
\begin{remark}\label{rk:canonicalclass}
The map $\theta_{\E}:\pi_1(\elp\G_m)\xto{\cong} \gal_{\E}^{\ab}$ gives us a map on corresponding cohomology groups:
$$\tH^2(\Gamma,\pi_1(\elp{\mathbb{G}_m}))\to\tH^2(\Gamma,\gal_{\E}^{\ab}).$$ The image of $u_{\E/\K}$ under this map (which we also denote by $u_{\E/\K}$ ) corresponds to the exact sequence:
     \begin{equation}
         1\to \gal_{\E}^{\ab}\to\mathscr{G}_{\E/\K}\to \Gamma \to 1.
     \end{equation}
      \end{remark}

    \section{Proofs of the Main Results}\label{sec:pfmain}
	\subsection{Proof of Theorem \ref{thm:main1}}\label{sec:pfmain1}
    In this section we compute the fundamental group of the pro-algebraic group $\lp{T}$ for a torus $T$ defined over $\K$. The fundamental group of the pro-algebraic group $\lp \G_m$ is known due to Serre (See \cite{Serre1961}). We use this result and the fact that the torus $T$ splits over a finite Galois extension $\E/\K$ to compute $\pi_1(\lp T)$. We first prove the following:
    \begin{lemma}{\label{gal}}
        Let $\Gamma$ be a finite group acting on a commutative pro-algebraic group $G$ defined over $k$. Then we have $\pi_1(G^{\Gamma})=\pi_1(G)^\Gamma$.
   \end{lemma} 
	\begin{proof} 
    
    \par The proof of this lemma is similar to that of \cite[\S 2.3 Prop. 4]{Serre1961}. An \(l\)-cochain is a function \(\Gamma^l\to G\). Let $C^l(\Gamma,G)$ be the group of $l$-cochains. This is a pro-algebraic group as it is isomorphic to a product of a finite number of copies of $G$. The boundary map $C^l(\Gamma, G) \xrightarrow{d} C^{l+1}(\Gamma,G)$ is a morphism of pro-algebraic groups, which gives us the following exact sequence: 
	\begin{equation}
	0\xrightarrow{} {G}^\Gamma \xrightarrow{} C^0(\Gamma,G) \xrightarrow{d} C^1(\Gamma,G).
	\end{equation}
	The functor $\pi_1$ is left-exact \cite[\S 10.2]{Serre1960}, and hence the following sequence is exact.
	\begin{equation}
	0\xrightarrow{} \pi_1({G}^\Gamma) \xrightarrow{} \pi_1(C^0(\Gamma,G)) \xrightarrow{\pi_1(d)} \pi_1(C^1(\Gamma,G)).
	\end{equation}
    The functor \(\pi_1\) commutes with products. The pro-algebraic \(C^l(\Gamma,G)\) is isomorphic to a product of \(|\Gamma|^l\)-copies of \(G\). Hence \(\pi_1(C^l(\Gamma, G))\cong C^l(\Gamma, \pi_1(G))\).
	The lemma now follows since the above exact sequence gives us the following exact sequence:
	\begin{equation}\label{eq:transfer}
	0\xrightarrow{} \pi_1(G^\Gamma) \xrightarrow{} C^0(\Gamma,\pi_1(G)) \xrightarrow{\pi_1(d)} C^1 (\Gamma, \pi_1(G)).	
	\end{equation}  
	\end{proof}
   \par We will now state and prove a more precise version of Theorem \ref{thm:main1}. Let $\E$ be a finite Galois extension of $\K$.
   Let $\E'/\K$ be a finite Galois extension containing $\E$, then $\gal_{\E'}$ is a finite index subgroup of $\gal_{\E}$ and we have the short exact sequence
     \begin{equation}
      0\to \gal_{\E'}\to \gal_{\E}\to \gal(\E'/\E)\to 0.
  \end{equation}
   We have the transfer homomorphism $\tran:\gal_{\E}^{\ab}\to \gal_{\E'}^{\ab}$ defined in \cite[Ch. VII \S 8]{Serre1967}. 
  
  The group $\gal_{\E}$ acts on the groups $\gal_{\E'}$ and $\gal_{\E}$ by conjugation and hence on their abelianizations. On $\gal_{\E'}^{\ab}$ and $\gal_{\E}^{\ab}$, the finite index normal subgroup $\gal_{\E'}$ of $\gal_{\E}$ acts trivially. Hence there is an action of the quotient $\gal(\E'/\E)$ on them. On $\gal_{\E}^{\ab}$ this action is trivial. By  \cite[\S 2.4, Prop. 7]{Serre1961} the transfer map is injective and gives an isomorphism
   $$\tran:\gal_{\E}^{\ab}\xto{\cong}\left({\gal_{\E'}^{\ab}}\right)^{\gal(\E'/\E)}\subset \gal_{\E'}^{\ab}.$$
    \par Let $T$ be a torus defined over $\K$ and split over a finite Galois extension $\E/\K$. Let $\E'/\K$ be a finite Galois extension containing $\E$. Then $\gal(\E'/\E)$ acts trivially on the co-character lattice $\cchg(T)$. Hence we have the isomorphism
    $$\tran:\cchg(T)\otimes\gal_{\E}^{\ab}\xto{\cong} (\cchg(T)\otimes{\gal_{\E'}^{\ab}})^{\gal(\E'/\E)}.$$
    Since the finite group $\gal(\E/\K)$ acts on both the domain and the target of the above map and it is a $\gal(\E/\K)$-module map, we have the following transfer isomorphism: 
  \begin{equation}\label{eq:traniso}
   \tran:({\cchg(T)\otimes\gal_{\E}^{\ab}})^{\gal(\E/\K)}\to (({{\cchg(T)\otimes\gal_{\E'}^{\ab})})^{\gal(\E'/\E)}})^{\gal(\E/\K)}=({\cchg(T)\otimes\gal_{\E'}^{\ab}})^{\gal(\E'/\K)}.   
  \end{equation}

   \begin{theorem}\label{thm:fundamentalgroup}
    For a torus $T$ defined over $\K$ and split over $\E$, we have a canonical isomorphism of topological groups:
    \begin{equation}
    \pi_1(\lp T)\xrightarrow[\cong]{\theta_{\E/\K}} {(\cchg(T)\otimes {\gal_{\E}^{\ab}})}^{\gal(\E/\K)}
    \end{equation}
    satisfying the following:\\
    (i) If $\E'/\K$ is a finite Galois extension containing $\E$, then the following diagram is commutative where $\tran$ is the transfer isomorphism \eqref{eq:traniso}:
    \begin{equation}\label{diag:transfer}
     \begin{tikzcd}
     & \pi_1(\lp T)\arrow[dl,"\theta_{\E/\K}"',"\cong"]\arrow[dr,"\theta_{\E'/\K}","\cong"'] & \\
    ({\cchg(T)\otimes\gal_{\E}^{\ab}})^{\gal(\E/\K)}\arrow[rr,"\tran","\cong"'] &        &({\cchg(T)\otimes\gal_{\E'}^{\ab}})^{\gal(\E'/\K)}.  
   \end{tikzcd}   
    \end{equation}
    (ii) (Functoriality.) Let $f:T_1\to T_2$ be a morphism of tori over $\K$. Let $\E/\K$ be a finite Galois extension over which both $T_1,T_2$ split. Then the following diagram is commutative:
    \begin{equation}\label{diag:natural}
        \begin{tikzcd}
            \pi_1(\lp T_1) \arrow[r,"\pi_1(f)"]\arrow[d,"\theta_{\E/\K}"] & \pi_1(\lp T_2)\arrow[d,"\theta_{\E/\K}"]\\
            (\cchg(T_1)\otimes\gal_{\E}^{\ab})^{\gal(\E/\K)}\arrow[r]& (\cchg(T_2)\otimes\gal_{\E}^{\ab})^{\gal(\E/\K)}.
        \end{tikzcd}
    \end{equation}
    (iii) Consider the induced torus $T'=\res_{\K'/\K}\G_m$ where $\K'/\K$ is a finite separable extension. In this case $\lp T'\cong\lp_{\K'} \G_m$ since the connected N\'eron model $\T'=\res_{\Ocal_{\K'}/\O}\G_m$. Let $\E/\K$ be a finite Galois extension containing $\K'$. Then  we have the following commutative diagram:
    \begin{equation}\label{diag:induced}
       \begin{tikzcd}
         \pi_1(\lp_{\K'} \G_m) \arrow[d,"\cong","\theta_{\K'}"']\arrow[r,"\cong"] &\pi_1(\lp T')\arrow[d,"\cong","\theta_{\E/\K}"']\\
         \gal_{\K'}^{\ab}\arrow[r,"\cong"]  & (\cchg(T')\otimes\gal_{\E}^{\ab})^{\gal(\E/\K)}.
           \end{tikzcd}
    \end{equation}

    \end{theorem}
    \begin{proof}
        \par We have that $T$ is defined over $\K$ and split over a finite Galois extension $\E/\K$. The action of $\gal(\E/\K)$ on $\T(\Ocal_{\E})\subseteq T(\E)$ induces a $\gal(\E/\K)$-action on the pro-algebraic group $\elp{T}$ as well as the full loop group $L_{\E}T$ such that $LT=(L_{\E}T)^{\gal(\E/\K)}$. Hence we have
        \begin{equation}\label{eqn:ltcirc}
        L^+T=(LT)^\circ=\left((L_{\E}T)^{\gal(\E/\K)}\right)^\circ=\left(\left((L_{\E}T)^\circ\right)^{\gal(\E/\K)}\right)^\circ=\left((L^+_{\E}T)^{\gal(\E/\K)}\right)^\circ.
        \end{equation}

    \noindent In fact, we will see in Remark \ref{rk:finitepi0} below that $\pi_0\left(L_{\E}^+(T)^{\gal(\E/\K)}\right)=\frac{L_{\E}^+(T)^{\gal(\E/\K)}}{L^+T}$ is finite. From Equation (\ref{eqn:ltcirc}) we have $$\pi_1(\lp{T})=\pi_1((({\elp{T})^{\gal(\E/\K)}})^{\circ})=\pi_1(({\elp{T})^{\gal(\E/\K)}}).$$ By Lemma \ref{gal} we get that 
        \begin{equation}\label{eq:pi1gamma}
            \pi_1(\lp{T})=\pi_1(\elp{T})^{\gal(\E/\K)}.
        \end{equation}
     Since $T$ is split over $\E$, the connected N\'eron $\Ocal_{\E}$-model of $T_{\E}$ is \(\text{Spec } \Ocal_{\E}[\cchg(T)]\). Hence $$\pi_1({\elp T})=\cchg(T)\otimes\pi_1(\elp\G_m).$$ By \cite[\S 2.5.10]{Serre1961} we have the canonical isomorphism of topological groups $\pi_1(\elp{\G_{m}})\xrightarrow[\cong]{\theta} {\gal_{\E}^{\ab}}.$ This gives the canonical isomorphism of topological groups: $$\pi_1(\elp{T})\xrightarrow[\cong]{\theta_{\E/\K}}\cchg(T)\otimes {\gal_{\E}^{\ab}}.$$ Now from \eqref{eq:pi1gamma} we obtain the first statement of the theorem.
   
    \par We now prove the commutativity of diagram \eqref{diag:transfer}. By \cite[\S 2.4 Prop. 7]{Serre1961} the following diagram is commutative:
    \begin{equation}\label{diag:Serre tran}
    \begin{tikzcd}
        \pi_1(\elp{\G_{m}})\arrow[r,hook]\arrow[d,"\theta_{\E/\K}"] & \pi_1(\lp_{\E'}{\G_{m}})\arrow[d,"\theta_{\E'/\K}"] & \\
      \gal_{\E}^{\ab}\arrow[r,"\tran"] & \gal_{\E'}^{\ab} & \end{tikzcd}
    \end{equation}
   We tensor the above square with $\cchg(T)$.  The map $\theta_{\E/\K}$ is a map of $\gal(\E/\K)$-modules hence we look at the $\gal(\E/\K)$-invariants of the domain and the target. We have $$\pi_1(\lp T)=(\cchg(T)\otimes\pi_1(\elp \G_m))^{\gal(\E/\K)}\mbox{ and }$$ $$\tran\left((\cchg(T)\otimes\gal_{\E}^{\ab})^{\gal(\E/\K)}\right)=({\cchg(T)\otimes\gal_{\E'}^{\ab}})^{\gal(\E'/\K)}.$$ Thus the commutativity of the previous diagram proves the commutativity of diagram \eqref{diag:transfer}.
   
    We now prove statement (ii). The map $f$ induces a map $\cchg(T_1)\to\cchg(T_2)$ of $\gal(\E/\K)$-modules, which we also denote by $f$. Since the connected N\'eron model is functorial, $f$ induces a map $\T_1\to\T_2$. This gives us a  map on the fundamental groups $\pi_1(f):\pi_1(\lp T_1)\to \pi_1(\lp T_2)$. 
    We then have the following commutative diagram:
    \begin{equation}
       \begin{tikzcd}
        \pi_1(\lp T_1)\arrow[r,"\pi_1(f)"]\arrow[d,equal] & \pi_1(\lp T_2)\arrow[d,equal]\\
        (\cchg(T_1)\otimes\pi_1(\elp \G_m))^{\gal(\E/\K)}\arrow[r,"f\otimes id"] & (\cchg(T_2)\otimes\pi_1(\elp \G_m))^{\gal(\E/\K)}.
    \end{tikzcd}  
    \end{equation}
    This proves the commutativity of diagram \eqref{diag:natural} as desired.

     We now prove statement (iii). We have $\gal(\E/\K')\leq\gal(\E/\K)$. Since $T'$ is the induced torus, we have $X^*(T')=\ind_{\gal(\E/\K')}^{\gal(\E/\K)}\Z$ as a $\gal({\E/\K})$-module. For any $\gal(\E/\K)$-module $M$ we have functorial identifications
     \begin{equation}
    \begin{split}
        M^{\gal(\E/\K')} & =\Hom_{\gal(\E/\K')}(\Z,M)\\
                         &=\Hom_{\gal{(\E/\K})}(\chg(T'),M)\cdots\mbox{by Frobenius reciprocity}\\
                         &=\Hom_{\gal(\E/\K)}(\Z,\cchg(T')\otimes M)\\
                         &=(\cchg(T')\otimes M)^{\gal(\E/\K)}.
    \end{split}
    \end{equation}
    Using the above for $M=\pi_1(\elp \G_m)$ and $M=\gal_{\E}^{\ab}$ gives us the commutativity of the right hand square  while the functoriality of $\theta$ (\cite[\S 2.4 Prop. 7]{Serre1961}) gives us the commutativity of the left square in following diagram:
    \begin{equation*}
    \begin{tikzcd}
        \pi_1(\lp_{\K'}{\G_{m}})\arrow[r,"\cong"]\arrow[d,"\theta_{\K'}"] & \pi_1(\elp \G_{m})^{\gal(\E/\K')}\arrow[d,"\theta_{\E}"]\arrow[r,"\cong"] & (\cchg(T')\otimes\pi_1(\elp \G_m))^{\gal(\E/\K)}\arrow[d,"\theta_{\E/\K}"]\\
      \gal_{\K'}^{\ab}\arrow[r,"\tran"',"\cong"] & (\gal_{\E}^{\ab})^{\gal(\E/\K')} \arrow[r,"\cong"]& (\cchg(T')\otimes\gal_{\E}^{\ab})^{\gal(\E/\K)}.
    \end{tikzcd}    
    \end{equation*}
    Statement (iii) then follows from the fact $\pi_1(\lp T')=\pi_1(\elp T')^{\gal({\E/\K})}=(\cchg(T')\otimes\pi_1(\elp \G_m))^{\gal(\E/\K)}.$
    \end{proof}
    \begin{remark}\label{rk:finitepi0}
    Since the torus $T$ splits over $\E$, we have the $\gal(\E/\K)$-equivariant short exact sequence
    \[0\to L^+_{\E}T=(L_{\E} T)^\circ\to L_{\E}T\to X_\ast(T)\to 0.\]
    Note that the norm map $T(\E)_{\gal(\E/\K)}\to T(\E)^{\gal(\E/\K)}=T(\K)$ is an isomorphism by the construction of the Kottwitz homomorphism $T(\K)\to X_*(T)_{\gal(\E/\K)}$ as described in \cite[Pf. of Prop. 11.1.1, pg. 448]{KP23}. Hence the Tate cohomology groups \(\widehat{H}^{-1}(\gal(\E/\K),T(\E))\) and \(\widehat{H}^{0}(\gal(\E/\K),T(\E))\) are trivial.
    The long exact sequence of Tate cohomology groups associated with the short exact sequence of $\gal(\E/\K)$-modules above gives us a natural identification\[\widehat{H}^0(\gal(\E/\K),T(\O_{\E}))\cong \widehat{H}^{-1}(\gal(\E/\K), \cchg(T)).\] The definition of these groups gives us \[T(\O_{\E})^{\gal(\E/\K)}/\nm(T(\O_{\E}))\cong\ker\left(\nm:\cchg(T)\to \cchg(T)^{\gal(\E/\K)}\right)\Big/\langle (1-\gamma)\chi|\gamma\in \gal(\E/\K),\chi\in \cchg(T)\rangle.\] 
    Now let \(n=|\gal(\E/\K)|\) and let \(\chi\in \cchg(T)\) with \(\nm(\chi)=\sum_{\gamma\in \gal(\E/\K)}\gamma\chi=0\). Then, \[n\chi=n\chi-\sum_{\gamma\in \gal(\E/\K)}\gamma\chi=\sum_{\gamma \in \gal(\E/\K)}(1-\gamma)\chi=0\in \widehat{H}^{-1}(\gal(\E/\K), \cchg(T)).\] In other words $\widehat{H}^{-1}(\gal(\E/\K), \cchg(T))$ is $n$-torsion as well as finitely generated, and hence finite. This implies \(T(\O_{\E})^{\gal(\E/\K)}/\nm(T(\O_{\K}))\) is finite and further implies  \(T(\O_{\E})^{\gal(\E/\K)}/(T(\O))\) is finite. From this we can deduce that in fact, $\pi_0\left(L_{\E}^+(T)^{\gal(\E/\K)}\right)=\widehat{H}^{-1}(\gal(\E/\K), \cchg(T))$ and that it is finite. Note also that $\pi_0(LT)=\pi_0\left((L_{\E}T)^{\gal(\E/\K)}\right)=X_*(T)_{\gal({\E/\K})}\supset \widehat{H}^{-1}(\gal(\E/\K), \cchg(T))$.
    \end{remark}
      
   \subsection{Proof of Theorem \ref{thm:main2} }\label{sec:pfmain2}

    In this section we will establish the inertial local Langlands correspondence for tori. For the multiplicative group $\G_m$ over $\K$, applying $\Hom_{\ell\text{-adic}}(\cdot,\mql)$ to \eqref{th} we get:
    $$\Hom_{\ell\text{-adic}}(\pi_1(\lp{\G_m}), \mql)\xto{\cong} \Hom_{\ell\text{-adic}}(\gal_{\K}^{\ab},\mql)=H^1_{\ell\text{-adic}}(\gal_{\K},{\G}_{m}(\ql)).$$
    Now consider induced tori. Let $\K'/\K$ be a finite separable extension. Consider the induced torus $T'=\res_{\K'/\K}\G_m$.  Theorem \ref{thm:fundamentalgroup}(iii) combined with the above gives us
    $$\Hom_{\ell\text{-adic}}(\pi_1(\lp{T'}),\mql)=\Hom_{\ell\text{-adic}}(\pi_1(\lp_{\K'}\G_m),\mql)\xrightarrow[\cong]{\theta_{\K'}}\Hom_{\ell\text{-adic}}(\gal_{\K'}^{\ab},\mql).$$
    On the other hand $\chg(T')=\ind_{\gal_{\K'}}^{\gal_{\K}}\chg(\G_m)=\Z[{\gal_{\K}}]\otimes_{\Z[\gal{\K'}]}\Z$ is the induced module as a $\gal_{\K}$-module. It is canonically isomorphic to the co-induced module $\Hom_{\gal_{\K'}}(\Z[{\gal_{\K}}],\chg(\G_m))$ as the index $[\gal_{\K}:\gal_{\K'}]$ is finite. Thus using Shapiro's lemma \cite[pg 117]{Serre1967} we have an isomorphism: 
    \begin{equation*}\label{eqn:Sh}
        H^1_{\ell\text{-adic}}(\gal_{\K},\widecheck{T'}(\ql))=H^1_{\ell\text{-adic}}(\gal_{\K},X^*(T')\otimes\mql)\xrightarrow[\sh]{\cong} H^1_{\ell\text{-adic}}(\gal_{\K'},\mql)=\Hom_{\ell\text{-adic}}(\gal^{\ab}_{\K'},\mql).
    \end{equation*}Combining with the previous equation we get the desired isomorphism:
      \begin{equation}\label{eqn:induced}
     \Hom_{\ell\text{-adic}}(\pi_1(\lp{T'}),\mql)\xto\cong H^1_{\ell\text{-adic}}(\gal_{\K},\widecheck{T'}(\ql)).   
      \end{equation}
    We now state and prove a more precise version of Theorem \ref{thm:main2} stated in the introduction. 
    Let $\mathfrak{T}$ denote the (additive) category of tori over $\K$.
    Consider the following two additive functors $\mathfrak{T}\to \mathcal{A}b^{\opp}$ given by
       \begin{equation}\label{eqn:functors}
        \begin{split}
        T\mapsto & \Hom_{\ell\text{-adic}}(\pi_1(\lp T),\mql)=\CS^+(T),\\
        T\mapsto &  H^1_{\ell\text{-adic}}(\gal_{\K},\dlt)=\{\mbox{inertial local Langlands parameters for }T\}.
      \end{split}
      \end{equation}
      \begin{theorem}\label{thm:inertial LLC}
        There is a unique additive natural isomorphism between the functors given in \eqref{eqn:functors}, i.e. given a torus $T$ over $\K$, we have a unique functorial isomorphism 
        $$\varphi_{T}:\CS^+(T)=\Hom_{\ell\text{-adic}}(\pi_1(\lp{T}),\mql)\xrightarrow{\cong} H^1_{\ell\text{-adic}}(\gal_{\K},\dlt)$$
        such that if $\E$ is a finite Galois extension of $\K$ and $T=\res_{\E/\K}\G_m$, then $\varphi_{T}$ is as obtained in \eqref{eqn:induced}. In fact for any finite separable extension $\K'$ of $\K$ and $T'=\res_{\K'/\K}\G_m$, the natural isomorphism $\varphi_{T'}$ is as given by \eqref{eqn:induced}.
    \end{theorem}
    
    The canonical isomorphism $\varphi_T$ is the inertial local Langlands correspondence for $\ql$-character sheaves on $\lp T$. This statement is similar to the statement of usual local Langlands correspondence for tori as given in \cite{JKYu[2004]}. We will give a proof of the above theorem along similar lines.
    We will first state and prove an auxiliary proposition before proving the theorem.
   \par Let $\E/\K$ be a finite Galois extension where $T$ splits and $\Gamma=\gal(\E/\K)$. From \eqref{galois} $\gal_{\E}^{\ab}$ is a finite index normal subgroup of $\mathscr{G}_{\E/\K}$. We choose $\{g_\gamma ,\gamma\in\Gamma\}$, a set of  right $\gal_{\E}^{\ab}$-coset representatives in $\mathscr{G}_{\E/\K}$. These are also the same as the left cosets. We have the corestriction map:
      \begin{equation}\label{eq:cormap}
\cor:\Hom(\gal_{\E}^{\ab},\dlt)\to H^1(\mathscr{G}_{\E/\K}, \dlt).
      \end{equation}
    It can be explicitly described as follows: For $\alpha\in \Hom(\gal_{\E}^{\ab},\dlt) $ and $g\in\mathscr{G}_{\E/\K}$
    \begin{equation}\label{eq:corestriction}
     \cor(\alpha)(g)=\sum_{\gamma\in\Gamma} g_{\gamma}\alpha(g_{\gamma}^{-1} g g_{\gamma'})  
     \end{equation} where $g_{\gamma'}$ is the representative such that $g_{\gamma}^{-1}g g_{\gamma'}\in \gal_{\E}^{\ab} $ (See \cite[\S 5]{Neu2008}). It can be shown that this map is independent of the choice of the coset representatives.

    \begin{proposition}\label{prop:cor}
	The corestriction map  \eqref{eq:cormap} gives the following isomorphisms of abelian groups: 
	 \begin{equation}
\Hom((\cchg(T)\otimes\gal_{\E}^{\ab})^{\Gamma},\mql)={\Hom(\gal_{\E}^{\ab},\dlt})_\Gamma \xrightarrow[\cong]{\cor} H^1(\mathscr{G}_{\E/\K}, \dlt),
	 \end{equation} 
	
       \begin{equation}
	 \Hom_{\ell\text{-adic}}((\cchg(T)\otimes\gal_{\E}^{\ab})^{\Gamma},\mql)={\Hom_{\ell\text{-adic}}(\gal_{\E}^{\ab},\dlt})_\Gamma \xrightarrow[\cong]{\cor} H^1_{\ell\text{-adic}}(\mathscr{G}_{\E/\K}, \dlt).
	 \end{equation} 
\end{proposition}
  We prove the following lemma before proving the proposition:
      \begin{lemma}\label{lemma:tH iso}
     	For each $r\in \Z$, we have an isomorphism: 
     	\begin{equation}
     	\tH^{r}(\Gamma, \Hom(\gal_{\E}^{\ab},\dlt))\to \tH^{r+2}(\Gamma,\dlt).   
     	\end{equation} 
     \end{lemma}
     \begin{proof}
     	We have the following pairing:
     	\begin{equation*}
     	\begin{split}
     	\Hom(\cchg(T)\otimes\gal_{\E}^{\ab},\mql) \times \gal_{\E}^{\ab}\to & \Hom(\cchg(T),\mql)=\dlt \\
            (f,c) \mapsto & f(c\otimes\_) :\cchg(T)\to\mql.
     	\end{split}
     	\end{equation*}
           We also have the natural pairing
           \begin{equation}
            \cchg(T)\times\gal_{\E}^{\ab}\to  \cchg(T)\otimes\gal_{\E}^{\ab}.   
           \end{equation}
            These are bilinear pairings of $\Gamma$-modules. These pairings and the canonical class $u_{\E/\K}\in H^2(\Gamma, \gal_{\E}^{\ab})$ gives us the following maps:
     	\begin{equation}
     	\begin{split}
     	y\mapsto y\cup u_{\E/\K}&:\tH^r(\Gamma,\Hom(\cchg(T)\otimes\gal_{\E}^{\ab},\mql))\to \tH^{r+2}(\Gamma,\dlt)\\
     	y\mapsto y\cup u_{\E/\K}&:\tH^r(\Gamma, \cchg(T))\to\tH^{r+2}(\Gamma,\cchg(T)\otimes\gal_{\E}^{\ab}).
     	\end{split}
     	\end{equation} 
     	Consider the following diagram:
     	\begin{equation*}
     	\begin{tikzcd}[row sep=large, column sep=tiny]
     	\tH^r(\Gamma,\Hom(\cchg(T)\otimes\gal_{\E}^{\ab},\mql))\arrow[d,"\_ \cup u_{\E/\K}"] &\times & \tH^{-r-1}(\Gamma,\cchg(T)\otimes\gal_{\E}^{\ab})\arrow{rrrr} & & & & \tH^{-1}(\Gamma,\mql)\arrow[d,equal]  \\
     	\tH^{r+2}(\Gamma,\dlt) &\times & \tH^{-r-3}(\Gamma, \cchg(T))\arrow[u,"u_{\E/\K} \cup\_"]\arrow{rrrr} & & & & \tH^{-1}(\Gamma, \mql).
     	\end{tikzcd}   
     	\end{equation*}
     	From \eqref{eq:canonicalclass}, we see that all the assumptions for Tate-Nakayama theorem \cite[Ch IX,\S8]{Serre1967} are satisfied. Hence, using the Tate-Nakayama theorem, the second vertical map is an isomorphism. The pairings are non-degenerate by \cite[Lemma 8.2]{Mil2006}. The diagram commutes by the associativity of cup products, hence the leftmost map is an isomorphism.  This proves the result.
     \end{proof}
    
     \begin{proof}[Proof of Proposition \ref{prop:cor}]
     We will complete this proof in several steps; it is similar to the proof of \cite[8.6]{Mil2006}.\\
     Step 1 shows that \(\Hom((\cchg(T)\otimes\gal_{\E}^{\ab})^{\Gamma},\mql)={\Hom(\gal_{\E}^{\ab},\dlt})_\Gamma\) and that the corestriction map factors through \({\Hom(\gal_{\E}^{\ab},\dlt})_\Gamma\).\\
     In step 2, we look at diagram \eqref{diag:cor} and prove that the corestriction map being an isomorphism is equivalent to the commutativity or anti-commutativity of the three smaller squares in the diagram. This is where we use Lemma \ref{lemma:tH iso}.\\
     Steps 3, 4, and 5 prove the commutativity/anti-commutativity of the individual squares of the diagram in step 2. This completes the proof of the fact that \({\Hom(\gal_{\E}^{\ab},\dlt})_\Gamma\xrightarrow{\cor} H^1(\mathscr{G}_{\E/\K}, \dlt)\) is an isomorphism.\\
     Step 6 proves that \(\cor\) then gives an isomorphism \({\Hom_{\ell\text{-adic}}(\gal_{\E}^{\ab},\dlt})_\Gamma\xrightarrow[\cong]{\cor} H^{1}_{\ell\text{-adic}}(\mathscr{G}_{\E/\K}, \dlt)\). We now begin the proof.

    \noindent
    \textbf{Step 1:}  Observe that since $\mql$ is a divisible abelian group,
      \begin{equation}
        \begin{split}
          \Hom((\cchg(T)\otimes\gal_{\E}^{\ab})^{\Gamma},\mql) & =  {\Hom(\cchg(T)\otimes\gal_{\E}^{\ab},\mql)}_{\Gamma}\\
          & =  \Hom(\gal_{\E}^{\ab},\chg(T)\otimes\mql)_{\Gamma}\\
          & =  {\Hom(\gal_{\E}^{\ab},\dlt})_\Gamma.
        \end{split}
    \end{equation}
       The fundamental group $\pi_1(\lp T)\subset\cchg(T)\otimes\pi_1(\elp\G_m)$ is a pro-abelian group. We know that $\pi_1(\elp\G_m)\cong \gal^{\ab}_{\E}$ has a Noetherian $\ell$-primary part (\cite[\S 4]{Serre1961}) and hence $\pi_1(\lp T)\cong(\cchg(T)\otimes\gal_{\E}^{\ab})^{\Gamma}$ has a Noetherian $\ell$-primary part. Then using Proposition \ref{cexact} we obtain the following equality in the same way as above: $$\Hom_{\ell\text{-adic}}((\cchg(T)\otimes\gal_{\E}^{\ab})^{\Gamma},\mql)={\Hom_{\ell\text{-adic}}(\gal_{\E}^{\ab},\dlt})_{\Gamma}.$$
    We now show that the map $\cor:\Hom(\gal_{\E}^{\ab},\dlt)\to H^1(\mathscr{G}_{\E/\K}, \dlt)$ factors through ${\Hom(\gal_{\E}^{\ab},\dlt)}_\Gamma $.
    
      Let $g\in \mathscr{G}_{\E/\K} $ and $\gamma_{\circ} \in \Gamma$. The action of $\Gamma$ on $\Hom(\gal_{\E}^{\ab},\dlt)$ is given by $(\gamma_{\circ}\alpha)(g)=g_{\gamma_{\circ}}\alpha(g_{\gamma_{\circ}}^{-1}g g_{\gamma_{\circ}})$.
    We have
      \begin{equation*}
       \cor({\gamma_{\circ}}\alpha)(g)=\sum_{\gamma\in\Gamma}g_\gamma g_{\gamma_{\circ}}\alpha(g_{\gamma_{\circ}}^{-1}g_\gamma^{-1}g g_{\gamma'}g_{\gamma_{\circ}})
      \end{equation*}
      with $(g_{\gamma'}g_{\gamma_{\circ}})^{-1}g g_{\gamma'}g_{\gamma_{\circ}}\in \gal_{\E}^{\ab}$. Since $\{g_\gamma g_{\gamma_{\circ}}|\gamma\in\Gamma\}$  form a system of coset representatives for $\Gamma$, $$\cor({\gamma_{\circ}}\alpha)=\cor(\alpha).$$ 
      This shows that $\cor$ factors through ${\Hom(\gal_{\E}^{\ab},\dlt)}_\Gamma $.
     
    \noindent
    \textbf{Step 2:} To prove that the corestriction map is an isomorphism, it is enough to show that each square in the following diagram commutes/anti-commutes: 
	\begin{small}
	\begin{equation}\label{diag:cor}
      \begin{tikzcd}[column sep=tiny]
        0\arrow [r] & {\tH^{-1}(\Gamma, \Hom(\gal_{\E}^{\ab},\dlt))} \arrow[r] \arrow[d, "\cong","\_\cup u_{\E/\K}
     	"'] & \Hom(\gal_{\E}^{\ab},\dlt)_\Gamma \arrow[r] \arrow[d, "\cor"] & \Hom(\gal_{\E}^{\ab},\dlt)^\Gamma \arrow[r] \arrow[equal]{d} & {\tH^{0}(\Gamma, \Hom(\gal_{\E}^{\ab},\dlt))} \arrow[d, "\cong","\_\cup u_{\E/\K}"'] \\ 
	     0 \arrow[r] & H^1(\Gamma, \dlt) \arrow[r,"\infl"] & H^1(\mathscr{G}_{\E/\K}, \dlt) \arrow[r,"\res"] & {H^1(\gal_{\E}^{\ab}, \dlt)}^\Gamma \arrow[r] & H^2(\Gamma, \dlt).
       \end{tikzcd}
       \end{equation}
	 \end{small}
     The top row of the diagram above is the definition of Tate cohomology groups. The bottom row comes from the inflation-restriction sequence \cite[Ch VII, \S 6]{Serre1967} corresponding to the short exact sequence of $\Gamma$-modules \eqref{eqn:exactseq}.
     The rightmost and leftmost isomorphisms are given by Lemma \ref{lemma:tH iso}.

    \noindent
    \textbf{Step 3:} We show that the right-most square of the diagram \eqref{diag:cor} anti-commutes. 
    
    We first describe the map $H^1(\gal_{\E}^{\ab},\dlt)\to H^2(\Gamma,\dlt)$ explicitly. Let $\alpha$ be a 1-cocycle representing a class in $H^1(\gal_{\E}^{\ab},\dlt).$ Extend it to a cochain $\beta:\mathscr{G}_{\E/\K}\to \dlt$ by defining 
    $$\beta(\rho g_{\gamma})=\alpha(\rho) \mbox{ for }\rho\in \gal_{\E}^{\ab}, \gamma\in \Gamma.$$ Then $d\beta$ is a 2-cocycle which is inflation of an element in $H^2(\Gamma,\dlt)$. This element of $H^2(\Gamma,\dlt)$ is the image of $\alpha$. This map is called the transgression map.
     \par We have $\mathscr{G}_{\E/\K}=\bigsqcup_{\gamma\in\Gamma} \gal_{\E}^{\ab} g_{\gamma}$ where $g_{\gamma}$ are right or left coset representatives. Using Remark \ref{rk:canonicalclass} we have, $$g_\gamma.g_{\gamma'}=u_{\E/\K}(\gamma,\gamma')g_{\gamma.\gamma'}.$$
      Let $\alpha \in {H^1(\gal_{\E}, \dlt)}^\Gamma$ and $d\beta$ be the image of $\alpha$ under the transgression map. Then, 
     \begin{equation*}
       \begin{split}
         d\beta(\gamma,\gamma') & =d\beta(g_\gamma,g_{\gamma'})\\
        &  =\gamma\beta(g_\gamma')-\beta(g_\gamma.g_{\gamma'})+\beta(g_{\gamma'})\\
          & = 0-\alpha(u_{\E/\K}(\gamma,\gamma'))+0 = -\alpha(u_{\E/\K}(\gamma,\gamma')).
       \end{split}
    \end{equation*}
    \noindent
   \textbf{Step 4:} We now prove that the left-most square of the diagram \eqref{diag:cor} commutes. We first show that $\cor$ maps an element of $\tH^{-1}(\Gamma, \Hom(\gal_{\E}^{\ab},\dlt))$ into the subgroup $ H^1(\Gamma, \dlt)$ of $H^1(\mathscr{G}_{\E/\K}, \dlt)$.
    For this, let $\alpha\in \Hom(\gal_{\E}^{\ab},\dlt) $, let $\rho\in \gal_{\E}^{\ab}$ and $g \in \mathscr{G}_{\E/\K}$. Then 
      \begin{equation*}
        \begin{split}
        \cor(\alpha)(\rho g) & = \sum_{\gamma\in\Gamma} g_{\gamma}\alpha(g_\gamma^{-1}\rho g_\gamma g_\gamma^{-1} g g_{\gamma'})\\
          & = \sum_{\gamma\in\Gamma}\gamma\alpha(\rho) + (\cor(\alpha))(g)\\
          & = \nm(\alpha)(\rho)+ \cor(\alpha)(g).
        \end{split}
      \end{equation*}
      If $\alpha\in \tH^{-1}(\Gamma, \Hom(\gal_{\E}^{\ab},\dlt))$ 
      then $\nm(\alpha)(\rho)=0$	 hence $\cor(\alpha)$ depends only on the coset of $g$ i.e. $\cor(\alpha)$ arises as inflation from an element of $H^1(\Gamma, \dlt)$.
      
       Now we show that the restriction of $\cor$ to $\tH^{-1}(\Gamma, \Hom(\gal_{\E}^{\ab},\dlt))$ is $\_\cup u_{\E/\K}$. Let $x\in\Gamma$ be a fixed element, as $\cor(\alpha)$ arises as inflation we can evaluate it at $x$. Since $g_xg_{x^{-1}\gamma}=u_{\E/\K}(x,x^{-1}\gamma)g_\gamma$, 
      $g_{\gamma'}$ in this case is $g_{x^{-1}\gamma}$.
     \begin{equation}\label{c1}
         \begin{split}
          \cor(\alpha)(x) &= \sum_{\gamma\in\Gamma} g_{\gamma}\alpha(g_{\gamma}^{-1} g_{x}g_{x^{-1}\gamma})\\
          &= \sum_{\gamma\in\Gamma} g_{\gamma}\alpha(g_{\gamma}^{-1}g_{x}g_{x^{-1}\gamma}g_{\gamma}^{-1}g_{\gamma})\\
           &= \sum_{\gamma\in\Gamma} (\gamma\alpha)(g_{x}g_{x^{-1}\gamma}g_{\gamma}^{-1})\\
           &= \sum_{\gamma\in\Gamma} (\gamma\alpha)(u_{\E/\K}(x,x^{-1}\gamma)).
         \end{split}
     \end{equation}
    
     As in Theorem 8.6 \cite{Mil2006} we use dimension shifting in order to make computations with cup products easier. We have the exact sequence:
     \begin{equation}
       0\to I_\Gamma\to\Z[\Gamma]\to\Z\to0.
     \end{equation}
     This gives that the following sequences are exact: 
     \begin{equation}
         0\to I_\Gamma\otimes \dlt\to\Z[\Gamma]\otimes \dlt\to\dlt\to0.
     \end{equation}
     \begin{equation}
          0\to I_\Gamma\otimes \Hom(\gal_{\E}^{\ab}, \dlt)\to\Z[\Gamma]\otimes  \Hom(\gal_{\E}^{\ab}, \dlt)\to  \Hom(\gal_{\E}^{\ab}, \dlt)\to0.
      \end{equation}
     Which gives the following diagram of connecting homomorphisms:
    \begin{equation}\label{diag:connectinghom}
	\begin{tikzcd}
	\tH^{-1}(\Gamma, \Hom(\gal_{\E}^{\ab},\dlt)) \arrow[r,"d^{-1}"]\arrow[d,"\_\cup u_{\E/\K}"] & \tH^{0}(\Gamma,\Hom(\gal_{\E}^{\ab},\dlt\otimes I_{\Gamma}))\arrow[d,"\_\cup u_{\E/\K}\otimes 1"] & \\
	\tH^1(\Gamma,\dlt)\arrow[r,"d^1"]& \tH^2(\Gamma,\dlt\otimes I_{\Gamma}).
	\end{tikzcd}
     \end{equation}
    The map $\_\cup u_{\E/\K}$ is the unique map that makes the diagram \eqref{diag:connectinghom} commute. Hence it is enough to show that the diagram \eqref{diag:connectinghom} commutes with $\cor$ in the place of $\_\cup u_{\E/\K}$. Let $\alpha\in \tH^{-1}(\Gamma, \Hom(\gal_{\E}^{\ab},\dlt))$ we have $\nm(\alpha)=0$ then,
    \begin{equation*}
     \begin{split}
        d^{-1}(\alpha) &=\nm(\alpha\otimes 1)\\
        &=\nm(\alpha\otimes 1)-\nm(\alpha)\otimes 1\\
        &=\sum_{\gamma\in\Gamma}\gamma\alpha\otimes(\gamma-1)
     \end{split}
     \end{equation*}
     and, 
     \begin{equation}
     d^1(\beta)(\gamma_1,\gamma_2)=\gamma_1\beta(\gamma_2)\otimes(\gamma_1-1)
     \end{equation}
     for $\gamma_1,\gamma_2\in\Gamma$.
     Now using \eqref{c1},
     \begin{equation}\label{eq:step4-1}
        (d^1\circ \cor\alpha)(\gamma_1,\gamma_2)=\sum_{\gamma\in\Gamma}\gamma_1.(\gamma\alpha)(u_{\E/\K}(\gamma_1,\gamma_2))\otimes(\gamma_1-1)
     \end{equation}
     and, 
     \begin{equation}\label{eq:step4-2}
      (d^{-1}\alpha\cup(u_{\E/\K}\otimes1))(\gamma_1,\gamma_2)=\sum_{\gamma\in\Gamma}\gamma\alpha(u_{\E/\K}(\gamma_1,\gamma_2))\otimes(\gamma-1).
     \end{equation}
     It is enough to show that the above equations are same.

     \noindent
    A co-chain in $C^1(\Gamma, \dlt\otimes I_{\Gamma})$ is of the form $\sum_{\gamma\in\Gamma}F_\gamma\otimes(\gamma-1)$ with $F_\gamma$ a map $\Gamma\to \dlt$.

     \noindent
     Let $\gamma_1,\gamma_2\in\Gamma$. A co-boundary in $B^2(\Gamma,\dlt\otimes I_\Gamma)$ can be written as: $$d \left(\sum_{\gamma\in\Gamma}F_\gamma\otimes(\gamma-1)\right)(\gamma_1,\gamma_2)=\sum_{\gamma\in\Gamma}(\gamma_1.F_{\gamma_1^{-1}\gamma}(\gamma_2)-F_\gamma(\gamma_1\gamma_2)+F_\gamma(\gamma_1))\otimes(\gamma-1)-\sum_{\gamma\in\Gamma}\gamma_1F_\gamma(\gamma_2)\otimes(\gamma_1-1).$$
     The above expression is obtained using:
     \begin{equation*}
       \begin{split}
       \gamma_1\left(\sum_{\gamma\in\Gamma}F_\gamma(\gamma_2)\otimes(\gamma-1)\right)&=\sum_{\gamma\in\Gamma}\gamma_1.F_\gamma(\gamma_2)\otimes(\gamma_1\gamma-\gamma_1)\\
        &= \sum_{\gamma\in\Gamma}\gamma_1.F_\gamma(\gamma_2)\otimes(\gamma_1\gamma-1)-\sum_{\gamma\in\Gamma}\gamma_1.F_\gamma(\gamma_2)\otimes(\gamma_1-1)\\
        &=\sum_{\gamma\in\Gamma}\gamma_1.F_{\gamma^{-1}\gamma}(\gamma_2)\otimes(\gamma-1)-\sum_{\gamma\in\Gamma}\gamma_1.F_\gamma(\gamma_2)\otimes(\gamma_1-1).
        \end{split}
     \end{equation*}
    Let $F_\gamma(\gamma_2)=(\gamma\alpha)(u_{\E/\K}(\gamma_2,\gamma_2^{-1}\gamma))$. Then,
    \begin{center}
	\begin{equation*}
	d\left(\sum_{\gamma\in\Gamma}F_\gamma\otimes(\gamma-1)-(d^{-1}\alpha)\cup(u_{\E/\K}\otimes 1)+d^1\circ\cor(\alpha)\right)(\gamma_1,\gamma_2)=
	\end{equation*}
    \end{center}

    \begin{equation*}
      \sum_{\gamma\in\Gamma}(\gamma\alpha)(\gamma_1 u_{\E/\K}(\gamma_2,\gamma_2^{-1}\gamma_1^{-1}\gamma).u_{\E/\K}(\gamma_1\gamma_2,\gamma_2^{-1}\gamma_1^{-1}\gamma)^{-1}.u_{\E/\K}(\gamma_1,\gamma^{-1}\gamma).u_{\E/\K}(\gamma_1,\gamma_2)^{-1}))\otimes (\gamma-1).
     \end{equation*}
     Using $h$ for $\gamma_2^{-1}\gamma_1^{-1}\gamma$ this becomes,
     \begin{equation*}
     \sum_{\gamma\in\Gamma}(\gamma\alpha)(\gamma_1 u_{\E/\K}(\gamma_2,h).u_{\E/\K}(\gamma_1\gamma_2,h)^{-1}.u_{\E/\K}(\gamma_1,\gamma_2 h).u_{\E/\K}(\gamma_1,\gamma_2)^{-1})\otimes(\gamma-1).
      \end{equation*}
      Each term in the above sum is zero because $u_{\E/\K}$ is a 2-cocycle. Therefore,
      \begin{equation*}
       d\left(\sum_{\gamma\in\Gamma}F_\gamma\otimes(\gamma-1)\right)=(d^{-1}\alpha)\cup(u_{\E/\K}\otimes 1)-(d^1\circ\cor(\alpha)).
       \end{equation*}
      This shows \eqref{eq:step4-1} and \eqref{eq:step4-2} are the same.

      \noindent
	 \textbf{Step 5:}
      To show that the middle square of the diagram \eqref{diag:cor} commutes, it is enough to show that
  	the composition:
  	\begin{equation*}
  	H^1(\gal_{\E}^{\ab},\dlt)\xrightarrow{\cor} H^1(\mathscr{G}_{\E/\K},\dlt)\xrightarrow{\res} H^1(\gal_{\E}^{\ab},\dlt) 
  	\end{equation*}
  	is the norm map ($\nm$)  for the action of $\Gamma$.
   \par This follows as, for $\alpha\in Z^1(\gal_{\E}^{\ab},\dlt)$ and $g\in\mathscr{G}_{\E/\K}$,
  	\begin{equation*}
  	\cor(\alpha)(g)=\sum_{\gamma\in\Gamma} g_{\gamma}\alpha(g_{\gamma}^{-1} g g_{\gamma'}).
  	\end{equation*}
  	If $g\in \gal_{\E}^{\ab}$ this becomes $\cor(\alpha)(g)=\sum_{\gamma\in\Gamma} g_{\gamma}\alpha(g_{\gamma}^{-1} g g_{\gamma})=\nm(\alpha)(g).$

    Thus all the above steps collectively show that 
     \begin{equation}
       \cor: {\Hom(\gal_{\E}^{\ab},\dlt})_\Gamma \to H^1(\mathscr{G}_{\E/\K}, \dlt)
      \end{equation}  
      is an isomorphism.\\
      \noindent
     \textbf{Step 6:}
      We will now prove that the following map is an isomorphism:
     \begin{equation}
       \cor: {\Hom_{\ell\text{-adic}}(\gal_{\E}^{\ab},\dlt})_\Gamma \to H^1_{\ell\text{-adic}}(\mathscr{G}_{\E/\K}, \dlt).
      \end{equation} 
    
     $\gal_{\E}^{\ab}$ is a finite index normal subgroup of $\mathscr{G}_{\E/\K}$. Thus for $f\in \Hom(\gal_{\E}^{\ab},\dlt)$, $\cor(f)$ is continuous iff its restriction to $\gal_{\E}^{\ab}$ is continuous i.e. $\res\circ\cor(f)$ is continuous. Step 5 then gives us, $\cor(f)$ is continuous iff $\nm(f)$ is continuous. 
    
     Since $\Hom_{\ell\text{-adic}}(\gal_{\E}^{\ab},\dlt)_{\Gamma}\cong\Hom_{\ell\text{-adic}}({(\cchg(T)\otimes\gal_{\E}^{\ab})}^\Gamma, \mql)$ as $\Gamma$- modules we have,
      \begin{equation*}
          \begin{split}
          \nm(f) \text{ is continuous} \Leftrightarrow & f|_{\nm(\cchg(T)\otimes\gal_{\E}^{\ab})} \text{ is continuous}. \\
          f|_{\nm(\cchg(T)\otimes\gal_{\E}^{\ab})} \text{ is continuous}\Leftrightarrow &
          f|_{{(\cchg(T)\otimes\gal_{\E}^{\ab})}^\Gamma} \text{ is continuous}.
          \end{split}
      \end{equation*}
      The last implication follows as $\nm(\cchg(T)\otimes \gal_{\E}^{ab})$ is a finite index subgroup of ${(\cchg(T)\otimes \gal_{\E}^{ab})}^\Gamma$ ( Prop. 4 \cite[\S 2.3]{Serre1961}). Thus we finally have,
      \begin{equation}
          \cor(f)\text{ is continuous} \Leftrightarrow f|_{{(\cchg(T)\otimes \gal_{\E}^{ab})}^\Gamma} \text{ is continuous}.
      \end{equation}
      \par Any $f\in \Hom_{\ell\text{-adic}}({(\cchg(T)\otimes \gal_{\E}^{ab})}^\Gamma,\mql) $ comes from a continuous element of $\Hom(\cchg(T)\otimes\gal_{\E}^{\ab},\mql)$ i.e. an element of $\Hom_{\ell\text{-adic}}(\cchg(T)\otimes\gal_{\E}^{\ab},\mql)$ using Proposition \ref{cexact}.  This completes the proof.
      
   \end{proof}   
    We will now prove the Theorem \ref{thm:inertial LLC}
    \begin{proof}[Proof of Theorem \ref{thm:inertial LLC}]
    Let $\E/\K$ be a finite Galois extension where the torus $T$ splits and $\Gamma$ denote the Galois group $\gal(\E/\K)$. Theorem \ref{thm:fundamentalgroup} gives us an isomorphism:
    \begin{equation*}
       \Hom_{\ell\text{-adic}}(\pi_1(\lp{T}),\mql)\cong\Hom_{\ell\text{-adic}}((\cchg(T)\otimes\gal_{\E}^{\ab})^\Gamma,\mql).
     \end{equation*}
     Proposition \ref{prop:cor} then gives 
      \begin{equation*}
       \Hom_{\ell\text{-adic}}(\pi_1(\lp{T}),\mql)\cong \Hom_{\ell\text{-adic}}((\cchg(T)\otimes\gal_{\E}^{\ab})^\Gamma,\mql)\xrightarrow[\cong]{\cor}H^1_{\ell\text{-adic}}(\mathscr{G}_{\E/\K},\dlt).
     \end{equation*}
      Since $\gal_{\E}$ acts trivially on $\dlt$, pull-back along the quotient map  $\gal_{\K}\to \mathscr{G}_{\E/\K}$ gives:
                 $$H^1_{\ell\text{-adic}}(\gal_{\K}, \dlt)=H^1_{\ell\text{-adic}}(\mathscr{G}_{\E/\K}, \dlt).$$ 
     This gives us the required isomorphism:
       \begin{equation*}
       \varphi_{T}:\Hom_{\ell\text{-adic}}(\pi_1(\lp{T}),\mql)\xrightarrow{\cong} H^1_{\ell\text{-adic}}(\gal_{\K},\dlt).
     \end{equation*}
     The isomorphism in the Theorem \ref{thm:main1} is functorial in $T$. It can also be seen that the corestriction map in the Proposition \ref{prop:cor} is functorial. Hence $\varphi_T$ is a morphism of functors.

     Let $\K'/\K$ be a finite separable extension. Let $\E/\K$ be a finite Galois extension such that $\E\supset\K'$. Let $T'=\res_{\K'/\K}\G_m$. Thus $\chg(T')=\ind_{\gal_{\K'}}^{\gal_{\K}}\chg(\G_m)$ as a $\gal_{\K}$-module. Hence $\widecheck{T'}(\ql)=\ind_{\gal_{\K'}}^{\gal_{\K}}\chg(\G_m)
     \otimes\mql$. We will now prove that the morphism $\varphi_{T'}$ is the same as that constructed in \eqref{eqn:induced}. Note that we have
     $$ H^1_{\ell\text{-adic}}(\gal_{\E},\widecheck{T'}(\ql))=\Hom_{\ell\text{-adic}}(\gal_{\E},\widecheck{T'}(\ql)),$$  $$H^1_{\ell\text{-adic}}(\gal_{\K'},\mql)=\Hom_{\ell\text{-adic}}(\gal_{\K'},\mql).$$
      Consider the diagram:
      \begin{equation}
         \begin{tikzcd}
         H^1_{\ell\text{-adic}}(\gal_{\E},\widecheck{T'}(\ql))\arrow[r,"\cor_{\gal_{\E}}^{\gal_{\K}}"]\arrow[d,"\cor_{\gal_{\E}}^{\gal_{\K'}}"]& H^1_{\ell\text{-adic}}(\gal_{\K},\widecheck{T'}(\ql))\arrow[d,equal]\\
         H^1_{\ell\text{-adic}}(\gal_{\K'},\widecheck{T'}(\ql))\arrow[r,"\cor_{\gal_{\K'}}^{\gal_{\K}}"]\arrow[d] & H^1_{\ell\text{-adic}}(\gal_{\K},\widecheck{T'}(\ql))\arrow[d,equal]\\
         H^1_{\ell\text{-adic}}(\gal_{\K'},\mql) & H^1_{\ell\text{-adic}}(\gal_{\K},\widecheck{T'}(\ql)).\arrow[l,"\sh"] 
         \end{tikzcd}
       \end{equation}
       The map \(\sh\) is due to Shapiro's lemma as in Equation \eqref{eqn:Sh}.
      The top part of the  diagram is commutative as $\cor_{\gal_{\E}}^{\gal_{\K}}=\cor_{\gal_{\K'}}^{\gal_{\K}}\circ\cor_{\gal_{\E}}^{\gal_{\K'}}$. The left-most map in the bottom part of the diagram is obtained using Frobenius reciprocity. The bottom part of the diagram is commutative as it is commutative for $0$-th cohomology groups(See \cite[\S 7.7]{JKYu[2004]}).\\
      This proves the required claim.
      
      The fact that such a $\varphi_T$ is uniquely determined can be proved using a proof similar to that in \cite[\S 7.6]{JKYu[2004]}.
    \end{proof}
    
   \subsection{Proof of Theorem \ref{thm:main3}}\label{sec:pfmain3}
    Theorem \ref{thm:main3} relates the usual local Langlands correspondence  \eqref{eq:LLCfortori} for tori over $K$ to the inertial local Langlands correspondence for tori over $\K$. We first relate the reciprocity map in local class field theory and the inertial local class field theory map constructed by Serre \cite{Serre1961}.
    \par Recall that here $K$ denotes a local field with finite residue field $\Fq$ and that $\K$ is the completion of maximal unramified extension of $K$. Let $E/K$ be a finite Galois extension. Hence $\E/\K$ is a finite Galois extension. Let $K'$ be the maximal unramified extension of $K$ inside $E$, hence $\E\supset \K'=\K$. The Galois groups  $\gal(\E/\K)$ and $\gal(E/K')$ are isomorphic via the restriction map. Denote  $\gal(\E/\K)=\gal(E/K')$ by $\Gamma$ 
    and the group $\gal(E/K)$ by $\wGamma$. Let the residue field of $E, K'$ be $\F_{q^n}$. The Galois group $\gal(K'/K)$ is isomorphic to the Galois group $\gal(\F_{q^n}/\Fq)$ and hence is a cyclic group $\langle \Fr \rangle$ of order $n$ generated by $\Fr$. We have the short exact sequence:
    \begin{equation}
        0\to\Gamma\to\wGamma\to \langle \Fr \rangle\to 0.
    \end{equation}
     The residue field of each of $\K,\K'$ and $\E$ is $k=\Fqcl$. 
       
    The kernel of the $\Fq$-Lang isogeny $\Lang_{K}:\lp\G_m\to \lp\G_m$ is $\lp \G_m(\Fq)=\Ocal_{K}^{\times}$.
    Hence by Lemma \ref{lemma:pi1to fqpts} we have a surjective homomorphism, $\p_{K}:\pi_1(\lp{\G_m})\twoheadrightarrow \Ocal_{K}^{\times}$. We have the following lemma:
    \begin{lemma}
    The following diagram is commutative
    \begin{center}
          \begin{tikzcd}
              \pi_{1}(\elp{\G_{m}})\arrow[r,"\p_{E}", two heads]\arrow[d,"\nm_{\Gamma}"]&\Ocal_{E}^{\times}\arrow[d,"\nm_{\Gamma}"]\\
              \pi_1(\lp{\G_m})\arrow[r,"\p_{K'} ", two heads]&\Ocal_{K'}^{\times}.
          \end{tikzcd}
      \end{center}  
     Here $\nm_{\Gamma}$ is the norm map for the group $\Gamma$. Note that $E/K'$ is a totally ramified finite Galois extension.
      \end{lemma}
       \begin{proof}
     Observe that $\elp \G_m(\F_{q^n})=\Ocal_{E}^{\times}$ and $\lp \G_m(\F_{q^n})=\Ocal_{K'}^{\times}$.  The norm map $\nm_{\Gamma}:\elp \G_m\to \lp \G_m$ is compatible with the $\F_{q^n}$-structures. Hence, the following diagram is commutative, and the lemma follows:
    \begin{equation}\label{eq:LangE/K}
      \begin{tikzcd}
      0\arrow[r] &\Ocal_{E}^{\times}\arrow[r]\arrow[d,"\nm_{\Gamma}"] &\elp \G_m \arrow[r,"\Lang_E"]\arrow[d,"\nm_{\Gamma}"] \arrow[r] & \elp \G_m \arrow[r]\arrow[d,"\nm_{\Gamma}"] &0\\
      0\arrow[r] &\Ocal_{K'}^{\times}\arrow[r] &\lp \G_m\arrow[r,"\Lang_{K'}"] &\lp \G_m \arrow[r] &0.
      \end{tikzcd}
    \end{equation}
    \end{proof}
    
    \noindent 
    The previous lemma gives us a homomorphism:
      $$\pi_1(\lp{\G_m})/\nm_{\Gamma}(\pi_1(\elp{\G_{m}}))\xrightarrow[]{{\p}_{K'}}\Ocal_{K'}^{\times}/\nm_{\Gamma}(\Ocal_{E}^{\times}).$$
    
    \noindent
     Thus for any finite Galois extension $E/K$, we have a homomorphism $$\gal(\E/\K)^{\ab}\xrightarrow[\cong]{\delta_{\E/\K}}\pi_1(\lp{\G_m})/\nm_{\Gamma}(\pi_1(\elp{\G_{m}}))\xrightarrow[]{{\p}_{K'}}\Ocal_{K'}^{\times}/\nm_{\Gamma}(\Ocal_{E}^{\times})\hookrightarrow K'^{\times}/\nm_{\Gamma}(E^{\times}).$$
      Recall that $\delta_{\E/\K}:\gal(\E/\K)^{\ab}\xrightarrow[\cong]{\delta_{\E/\K}}\pi_1(\lp{\G_m})/\nm_{\Gamma}(\pi_1(\elp{\G_{m}}))$ is as in \eqref{eq:defn-delta}. We then have the following lemma:
     \begin{lemma}\label{lemma:Dwork}
     Let $E/K$ be a finite abelian Galois extension. Let $K'$ be as mentioned above. 
     Then the following diagram commutes:
     \begin{center}
         \begin{tikzcd}
            \gal(\E/\K)\arrow[r,"\delta_{\E/\K}","\cong"']\arrow[d,"\res","\cong"'] & \pi_1(\lp{\G_m})/\nm_{\Gamma}(\pi_1(\elp{\G_{m}}))\arrow[d,"\p_{K'}"]\\
            \gal(E/K')\arrow[r,"\re^{-1}","\cong"'] & K'^{\times}/\nm_{\Gamma}(E^{\times}).
         \end{tikzcd}
     \end{center}
     Here $\re$ denotes the reciprocity map \eqref{rec}.
     \end{lemma}
   \begin{proof}
     The extension $E/K'$ is a totally ramified abelian extension. Consider the diagram \ref{eq:LangE/K}. We use the following description of the map $\re^{-1}$ given by Dwork's theorem, i.e. \cite[Ch.XIII, \S5, Cor. to Thm. 2]{Serre1967} (see also \cite{Dwork1958}): Let $\sigma\in \Gamma=\gal(\E/\K)=\gal(E/{K'})$. Then $\sigma(\varpi_{\E})\varpi_{\E}^{-1}\in \Ocal_{\E}^\times$ where $\varpi_{\E}$ is a uniformizer element in $\Ocal_{\E}$. By Lang's theorem, there exists a $y\in \Ocal_{\E}^\times$ such that $\Lang_E(y)=y\Fr_{E}(y^{-1})=\sigma(\varpi_{\E})\varpi_{\E}^{-1}$. Then by {\it loc. cit.}, $\re^{-1}(\sigma)$ is represented by $\nm_\Gamma(y)\in \left(\Ocal_{\K}^\times\right)^{\Fr_{K'}}=\Ocal_{K'}^\times\subset {K'}^\times$. In other words $\re^{-1}(\sigma)$ is given by the coboundary map associated with the diagram \ref{eq:LangE/K} applied to the element $\sigma(\varpi_{\E})\varpi_{\E}^{-1}$ which lies in the kernel of the rightmost vertical map in diagram \ref{eq:LangE/K}.

     Now $\delta_{\E/\K}(\sigma)$ is defined in \cite[\S2.5]{Serre1961} in the same way by applying the coboundary map (to the same element $\sigma(\varpi_{\E})\varpi_{\E}^{-1}$ as above), but now associated with the analogous diagram
     \begin{equation}\label{eq:UniversalcoverE/K}
      \begin{tikzcd}
      0\arrow[r] &\pi_1(L^+_{\E}\G_m)\arrow[r]\arrow[d,"\nm_{\Gamma}"] &\overline{\elp \G_m} \arrow[r,"\pi_{\E}"]\arrow[d,"\nm_{\Gamma}"] \arrow[r] & \elp \G_m \arrow[r]\arrow[d,"\nm_{\Gamma}"] &0\\
      0\arrow[r] &\pi_1(L^+\G_m)\arrow[r] &\overline{\lp \G_m}\arrow[r,"\pi_{\K}"] &\lp \G_m \arrow[r] &0
      \end{tikzcd}
    \end{equation} involving the universal covers $\pi_{\E},\pi_{\K}$ in place of the $\F_{q^n}$-Lang isogenies $L_E,L_{K'}$ in diagram \ref{eq:LangE/K}. Since the universal covers factor through the Lang isogenies, we have a commutative diagram combining \ref{eq:LangE/K},\ref{eq:UniversalcoverE/K} with natural maps from the objects in \ref{eq:UniversalcoverE/K} to the corresponding objects in \ref{eq:LangE/K}. Hence the coboundary map associated with \ref{eq:LangE/K} equals the coboundary map associated with \ref{eq:UniversalcoverE/K} composed with $\p_{K'}$ as desired.
    \end{proof}

     We can now compare the maps \ref{rec} and \ref{th}. We use $\delta$ to denote the map: $\gal_{\K}^{\ab}\to\pi_1(\lp\G_m)$ induced by the system of maps $\{\delta_{\E/\K}\}.$ Recall that $\delta=\theta^{-1}$ where $\theta$ is the map defined by Serre i.e. \eqref{th}.
     \begin{lemma}
      The following diagram is commutative:
     \begin{equation}\label{diagram:Serre reciprocity}
        \begin{tikzcd}
          \gal_{\K}^{\ab} \arrow[r,
            "\delta=\theta^{-1}","\cong"'] \arrow[d] & \pi_1(\lp \mathbb{G}_m) \arrow[d,"\p_K"] \\
          \gal_{K}^{\ab} & K^\times\arrow[l,"\re"'].
        \end{tikzcd}
     \end{equation}
     \end{lemma}
     \begin{proof}
     Let $E/K$ be a finite abelian Galois extension. Let $K'$ be maximal unramified extension of $K$ inside $E$. Consider the following diagram:
     \begin{center}
        \begin{tikzcd}
          \gal({\E}/{\K}) \arrow[r,"\delta_{\E/\K}"] \arrow[d,"\res","\cong"'] & \pi_1(\lp{\G_m})/\nm_{\Gamma}(\pi_1(\elp{\G_{m}}))\arrow[d,"\p_{K'}"]\\
           \gal({E}/{K'})\arrow[r,"\re^{-1}"]\arrow[d,hook] & K'^\times/\nm_{\Gamma}(E^\times)\arrow[d,"\nm_{\wGamma/\Gamma}"]\\
          \gal({E}/{K})\arrow[r,"\re^{-1}"] & K^\times/\nm_{\Gamma}(E^\times).
        \end{tikzcd}
     \end{center}
      The top part of the diagram is commutative by the previous lemma. The bottom part of the diagram is commutative by functorial properties of reciprocity map (See \cite[Ch XI]{Serre1967}).
     Thus the above diagram is commutative for any finite abelian Galois extension $E$ of $K$. Observe that the map ${\p}_K:\pi_1(\lp{\G_m})\twoheadrightarrow \Ocal_{K}^{\times}$ (defined using Lemma \ref{lemma:pi1to fqpts}) is the same as the composition:
     \begin{equation}\label{eq:defnp_K}
         \pi_1(\lp{\G_m})\twoheadrightarrow\Ocal_{K'}^{\times}\twoheadrightarrow \Ocal_{K}^{\times}.
     \end{equation}
     Here the second map is $\nm_{\wGamma/\Gamma}$ i.e. the norm map for the quotient group $\wGamma/\Gamma=\langle\Fr \rangle$. The commutativity of the diagram in the statement then follows.
      \end{proof}
      Consider the torus $T$ defined over $K$ and split over a finite Galois extension $E/K$.
      Theorem \ref{thm:sh-fn} for the pro-algebraic group $\lp{T}$ gives the following isomorphisms:
     \begin{equation*}
     \Hom_{\sm}(\lp{T}(\mathbb{F}_{q}),\mql)=
     \Hom_{\ell\text{-adic}}(\lp{T}(\mathbb{F}_{q}),\mql)\xrightarrow{\cong}  \Hom_{\ell\text{-adic}}(\pi_{1}(\lp{T}),\mql)^{\Fr}. \\
     \end{equation*} 
      The above isomorphism then gives us:
      \begin{equation*}
     \Hom_{\sm}(\T(\Ocal),\mql)=
     \Hom_{\ell\text{-adic}}(\T(\Ocal),\mql)\xrightarrow{\cong}  \Hom_{\ell\text{-adic}}(\pi_{1}(\lp{T}),\mql)^{\Fr}. \\
     \end{equation*} 
     Next, taking $\Fr$-invariants of the statement of Theorem \ref{thm:inertial LLC} gives:
     \begin{equation*}
       \Hom_{\ell\text{-adic}}(\pi_1(\lp{T}),\mql)^{\Fr}\xrightarrow[\varphi_T]{\cong}  H^1_{\ell\text{-adic}}(\gal_{\K}, \dlt)^{\Fr}.
     \end{equation*}
     Thus we have the following isomorphism:
      \begin{equation}\label{eq:inertialLLC}
      \Hom_{\sm}( \T(\Ocal),\mql)=\Hom_{\ell\text{-adic}}( \T(\Ocal),\mql)\xrightarrow[\psi]{\cong} H^1_{\ell\text{-adic}}(\gal_{\K},\dlt)^{\Fr}.     
      \end{equation}
     The usual local Langlands correspondence can also be stated as:
     \begin{equation}\label{eq:ql-LLC}
      \Hom_{\sm}(T(K),\mql)=\Hom_{\ell\text{-adic}}(T(K),\mql)\xrightarrow[\phi]{\cong} H^1_{\ell\text{-adic}}(W_{K},\dlt)
     \end{equation}
     Notice that $\dlt$ has the $\ell$-adic topology. This does not make a difference as mentioned in the introduction. 
     Having stated these correspondences we prove the following proposition which proves the commutativity of the second square in Theorem \ref{thm:main3}. 

     \begin{proposition}\label{propo:LLC}
     The following diagram commutes.
     \begin{center}
        \begin{tikzcd}
         \Hom_{\sm}( \T(\Ocal),\mql)\arrow[r,"\psi","\cong"'] & H^1_{\ell\text{-adic}}(\gal_{\K},\dlt)^{\Fr}\\
         \Hom_{\sm}(T(K),\mql)\arrow[r,"\phi","\cong"']\arrow[u,"\res"] & H^1_{\ell\text{-adic}}(W_{K},\dlt)\arrow[u,"\res"]
        \end{tikzcd}
     \end{center}
     Here $\phi$ is the local Langlands correspondence \eqref{eq:ql-LLC}, $\psi$ is the map in \eqref{eq:inertialLLC} and $\res$ is the restriction map.
     \end{proposition}
     We shall prove the proposition after proving the next lemma.\\
     We have the homomorphisms  $\p_E:\pi_1(\elp\G_m)\to \Ocal_{E}^{\times}\hookrightarrow E^{\times}$ (Lemma \ref{lemma:pi1to fqpts}) and $\gal_{\E}\hookrightarrow W_E$. These give us:
     \begin{equation}
        \begin{split}
            \Hom_{\ell\text{-adic}}(E^\times, \dlt) &\xrightarrow{\p_E}  \Hom_{\ell\text{-adic}}(\pi_1(\elp \mathbb{G}_m),\dlt)_{\Gamma}\\
            \Hom_{\ell\text{-adic}}(W_{E}^{\ab},\dlt)=\Hom_{\ell\text{-adic}}(W_E,\dlt)& \xrightarrow{\res} \Hom(\gal_{\E},\dlt)_{\Gamma}=\Hom_{\ell\text{-adic}}(\gal_{\E}^{\ab},\dlt)_{\Gamma}.
        \end{split}
     \end{equation}
     Abusing the notation, we continue to use $\p_E$ to denote the map induced by $\p_E$. We then have:
     \begin{equation}
     \begin{split}
          \Hom_{\ell\text{-adic}}(E^\times, \dlt)_{\wGamma} \xrightarrow{\nm_{\wGamma/\Gamma}}  \Hom_{\ell\text{-adic}}(E^\times, \dlt)_{\Gamma}\xrightarrow{\p_E} \Hom_{\ell\text{-adic}}(\pi_1(\elp \mathbb{G}_m),\dlt)_{\Gamma} \\
          \Hom_{\ell\text{-adic}}(W_{E}^{\ab},\dlt)_{\wGamma}  \xrightarrow{\nm_{\wGamma/\Gamma}} \Hom_{\ell\text{-adic}}(W_{E}^{\ab},\dlt)_{\Gamma}\xrightarrow{\res}\Hom_{\ell\text{-adic}}(\gal_{\E}^{\ab},\dlt)_{\Gamma}.
     \end{split}
     \end{equation}
     We denote the above maps by $\alpha$ and $\beta$ respectively, i.e.
     \begin{equation}\label{eq:defnalpha}
       \begin{split}
          \alpha := & \nm_{\wGamma/\Gamma}\circ\p_E\\
        \beta := & \nm_{\wGamma/\Gamma}\circ\res.
       \end{split}
     \end{equation}
     \begin{lemma}\label{lemma:diag2commutativity}
     The homomorphisms $\alpha$ and $\beta$ defined in \eqref{eq:defnalpha} make the following diagram commute. 
     \begin{equation}\label{Diagram2}
        \begin{tikzcd}
         (\Hom_{\ell\text{-adic}}(\pi_1(\elp \mathbb{G}_m),\dlt)_{\Gamma})^{\Fr}\arrow[r,"\theta_{\E}^{-1}"]& (\Hom_{\ell\text{-adic}}(\gal_{\E}^{\ab},\dlt)_{\Gamma})^{\Fr}\\
          \Hom_{\ell\text{-adic}}(E^\times, \dlt)_{\wGamma} \arrow[r,"\re_{E}^{-1}"]\arrow[u,"\alpha"] & \Hom_{\ell\text{-adic}}(W_{E}^{\ab},\dlt)_{\wGamma}\arrow[u,"\beta"].
        \end{tikzcd}
     \end{equation}
     
     Here $\re_{E}^{-1}$ denotes the map induced by the inverse of the reciprocity map for the field $E$. The map $\theta_{\E}$ is the map \ref{th} for the field $\E$. 
    \end{lemma}
    \begin{proof}
     The definition of the homomorphisms $\alpha$ and $\beta$ gives that, $$\im(\alpha)\subset(\Hom_{\ell\text{-adic}}(\pi_1(\elp \mathbb{G}_m),\dlt)_{\Gamma})^{\Fr}$$ and $$\im(\beta)\subset(\Hom_{\ell\text{-adic}}(\gal_{\E}^{\ab},\dlt)_{\Gamma})^{\Fr}.$$
      The commutativity of this diagram then follows by commutativity of the diagram \eqref{diagram:Serre reciprocity} for the field $E$ and the fact that $\delta$ is the inverse of the map $\theta$ constructed by Serre.
     \end{proof}    
     \noindent We now prove the proposition \ref{propo:LLC},  
     \begin{proof}[Proof of Proposition \ref{propo:LLC}]
     We have the following diagram:
     \begin{equation*}
    \scriptstyle
      \begin{tikzcd}[column sep=small]
       \Hom_{\ell\text{-adic}}( \T(\Ocal),\mql)\arrow[r,"\gamma"'] & (\Hom_{\ell\text{-adic}}(\pi_1(\elp\mathbb{G}_m),\dlt)_\Gamma)^{\Fr}\arrow[r,"\scriptscriptstyle \theta_{E}^{-1}"']& (\Hom_{\ell\text{-adic}}(\gal_{\E}^{\ab},\dlt)_{\Gamma})^{\Fr} \arrow[r,"\scriptscriptstyle \cor_{\Gamma}"']& H^{1}(\gal_{\K},\dlt)^{\Fr}\\
       \Hom_{\ell\text{-adic}}(T(K),\mql)\arrow[r]\arrow[u,"\res"] & \Hom_{\ell\text{-adic}}(E^\times, \dlt)_{\widetilde{\Gamma}}\arrow[u,"\alpha"]\arrow[r,"\re_{E}^{-1}"]\arrow[u,"\alpha"] & \Hom_{\ell\text{-adic}}(W_{E}^{\ab},\dlt)_{\wGamma}\arrow[u,"\beta"]\arrow[r,"\scriptscriptstyle \cor_{\wGamma}"]\arrow[u,"\beta"] & H^1(W_{E/K},\dlt)\arrow[u].
      \end{tikzcd}
     \end{equation*}
     \textbf{Step 1:} Consider the left-most square in the diagram. Since the residue field of $K'$ is $\mathbb{F}_{q^n}$, $\lp{T}(\mathbb{F}_{q^n})=\T(\Ocal_{K'})$. Define $\gamma$ as the following composition:
        \begin{equation}
        \begin{split}
        \Hom_{\ell\text{-adic}}( \T(\Ocal),\mql) \xrightarrow{\widehat{\nm}} \Hom_{\ell\text{-adic}}(\T(\Ocal_{K'}),\mql)^{\Fr} \to (\Hom_{\ell\text{-adic}}(\pi_1(\lp T),\mql)^{\Fr_{K'}})^{\Fr}=\\
         \Hom_{\ell\text{-adic}}(\pi_1(\elp T)^{\Gamma},\mql)^{\Fr} \to (\Hom_{\ell\text{-adic}}(\pi_1(\elp T), \mql)_{\Gamma}) ^{\Fr} \to (\Hom_{\ell\text{-adic}}(\pi_1(\elp\mathbb{G}_m) ,\dlt)_{\Gamma}) ^{\Fr}.
     \end{split}
     \end{equation}
     The first and second map is obtained using Theorem \ref{thm:sh-fn}. The  third map is obtained using Theorem \ref{thm:main1}. \\
     It can be seen using Theorem \ref{thm:sh-fn} that the composition of all maps in the top row of the above diagram is $\psi$ (\ref{eq:inertialLLC}). \\
     Recall that $\alpha=\nm_{\wGamma/\Gamma}\circ\p_E$  \eqref{eq:defnalpha}. 
     The map in Theorem \ref{thm:sh-fn} is obtained by pulling back along $\p_E$, hence it follows that this square commutes. 

     \noindent
     \textbf{Step 2:} We now consider the rightmost square. The inclusion map $\gal_{\K}\hookrightarrow W_K$ induces a map $\mathscr{G}_{\E/\K}\to W_{E/K}$. This gives us a map $H^1(W_{E/K},\dlt)\to H^1(\mathscr{G}_{\E/\K},\dlt).$ The right-most square can be expanded as
     \begin{equation}\label{diag:square2expanded}
       \begin{tikzcd}
      {(\Hom_{\ell\text{-adic}}(\gal_{\E}^{\ab},\dlt)_{\Gamma})}^{\Fr}\arrow[r,"\cor_{\Gamma}"] & H^1_{\ell\text{-adic}}(\mathscr{G}_{\E/\K},\dlt)^{\Fr}\arrow[r,"\sim"] & H^1_{\ell\text{-adic}}(\gal_{\K},\dlt)^{\Fr}\\
      \Hom_{\ell\text{-adic}}(W_{E}^{\ab},\dlt)_{\wGamma}\arrow[r,"\cor_{\wGamma}"]\arrow[u,"\beta"] & 
       H^1_{\ell\text{-adic}}(W_{E/K},\dlt)\arrow[u] \arrow[r,"\sim"] & 
       H^1_{\ell\text{-adic}}(W_{K},\dlt)\arrow[u,"\res"].
       \end{tikzcd}
     \end{equation}
      Here $\cor_{\Gamma}$ and $\cor_{\wGamma}$ denote the corestriction map for $\Gamma$ and $\wGamma$ respectively. Recall that $\beta$ is as defined in \eqref{eq:defnalpha}.\\
      The group \(\Gamma\) is isomorphic to the quotient \(\mathscr{G}_{\E/\K}/\gal_{\E}^{\ab}\) and the group \(\wGamma\) is isomorphic to the quotient \(W_{E/K}/ W_{E}^{\ab}\).
      If $\{g_{\gamma}|\gamma\in\Gamma \}$ are the coset representatives for \(\gal_{\E}^{\ab}\) in $\mathscr{G}_{\E/\K}$, then $\{\Fr^i.g_{\gamma}|\gamma\in\Gamma, 0\leq i \leq n \}$  are the coset representatives for $W_{E}^{\ab}$ in \(W_{E/K}\). This shows that the first square in the diagram \eqref{diag:square2expanded} above is commutative and hence the diagram \eqref{diag:square2expanded} is commutative. 
     This shows that the rightmost square commutes.\\
     The commutativity of the middle square follows from the Lemma \ref{lemma:diag2commutativity}. 
     This proves the proposition. 
     \end{proof}
    
      \subsubsection{On the Kottwtiz homomorphism for tori}\label{sec:kottwitz}
     Let $T$, $K$ and $\Ocal$ be as in the previous section. Let $E/K$ be a finite Galois extension where $T$ splits.  We recall the definition of the Kottwitz homomorphism as given in \cite{Kottwitz[1997]}. Let $\E$ denote the completion of maximal unramified extension of $E$. We have the valuation  map $\E^{\times}\to\Z$ which gives ,
      ${\E}^{\times}\otimes\cchg(T)\to\cchg(T)$. Observe that this map is $\Gamma=\gal(\E/\K)$-equivariant. Using ${\E}^{\times}\otimes\cchg(T)=T(\E)$ we get a map
      $$T(\E)_{\Gamma}\to\cchg(T)_{\Gamma}. $$
     It can be shown that $T(\E)_{\Gamma}\cong T(\K)$ (\cite[\S 11.1]{KP23}). Passing to the $\Fr$-invariants we have a homomorphism 
     $$\kappa:T(K)\to(\cchg(T)_{\Gamma})^{\Fr}.$$ This is called the Kottwitz homomorphism. It is continuous, surjective, with kernel $\T(\Ocal)$ where $\T$ is the connected N\'eron model of $T$. It can also be shown that its construction is independent of the choice of $E$ (\cite[\S 11.1]{KP23}). The Kottwitz homomorphism along with Lang's theorem gives us the following short exact sequence:
      \begin{equation}
          0\to \ \T(\Ocal)\to T(K)\xrightarrow[]{\kappa} (\cchg(T)_{\Gamma})^{\Fr}\to 0.
      \end{equation}
      We now apply the functor $\Hom_{\ell\text{-adic}}(\cdot ,\mql)$ to the above short exact sequence. $T(K)$ is a locally profinite group. $\T(\Ocal)$ is a subgroup of $T(K)$ with $T(K)/\T(\Ocal)$ discrete. Thus any character of $\T(\Ocal)$ can be extended to $T(K).$ Also, a character is continuous on $T(K)$ iff it is continuous on $\T(\Ocal)$. Hence the following sequence is exact: 
      \begin{equation}\label{eq:Kott dual}
          0\to\Hom_{\ell\text{-adic}}((\cchg(T)_{\Gamma})^{\Fr},\mql)\to\Hom_{\ell\text{-adic}}(T(K),\mql)\to \Hom_{\ell\text{-adic}}(\ \T(\Ocal),\mql)\to 0.
      \end{equation}
     We have the short exact sequence, 
     \begin{equation}
            0\to\gal_{\K}\to W_K \to \Z \to 0.
     \end{equation}
     Consider the inflation-restriction sequence obtained from above sequence:
     \begin{equation}\label{eq:inf-res}
            0\to H^1(\Z,\dlt^{\gal_{\K}})\to H^1_{\ell\text{-adic}}(W_K,\dlt)\to H^1_{\ell\text{-adic}}(\gal_{\K},\dlt)^{\Fr}.
     \end{equation}
     \begin{lemma}\label{lemma:shortexactiden}
     There is a canonical identification: $$\Hom_{\ell\text{-adic}}((\cchg(T)_{\gal_{\K}})^{\Fr},\mql)\xrightarrow[]{\cong}  H^1(\Z,\dlt^{\gal_{\K}}).$$
     \end{lemma}
     \begin{proof}
     Since $(\cchg(T)_{\gal_{\K}})^{\Fr}$ is a discrete group, $\Hom_{\ell\text{-adic}}((\cchg(T)_{\gal_{\K}})^{\Fr},\mql)=\Hom((\cchg(T)_{\gal_{\K}})^{\Fr},\mql)$. Thus:
     \begin{equation}\label{eq:tailcalculation}
        \begin{split}
         \Hom_{\ell\text{-adic}}((\cchg(T)_{\gal_{\K}})^{\Fr},\mql) &= \Hom(\cchg(T)_{\gal_{\K}},\mql)_{\Fr}\\
          &=\Hom(\cchg(T),\mql)^{\gal_{\K}}/(\Fr-id)\cdot\Hom(\cchg(T),\mql)^{\gal_{\K}} \\
          & = \Hom(\Z,\dlt)^{\gal_{\K}}/(\Fr-id)\cdot\Hom(\Z,\dlt)^{\gal_{\K}}\\
          & = \dlt^{\gal_{\K}}/(\Fr-id)\cdot\dlt^{\gal_{\K}} = H^1(\Z,\dlt^{\gal_{\K}}).
        \end{split}
     \end{equation}
     \end{proof}
     \noindent We now complete the proof of Theorem \ref{thm:main3}
    
     \begin{proof}[Proof of Theorem \ref{thm:main3}]
    Using Proposition \ref{propo:LLC}, it follows that the second square in the diagram in the statement of Theorem \ref{thm:main3} commutes.
     We now show that the first square commutes. The left-most isomorphism is the one constructed in Lemma \ref{lemma:shortexactiden}. We have the short exact sequence:
     $$1\to W_{E}^{\ab}\to W_{E/K}\to \wGamma\to 1.$$ Hence we write 
     $W_{E/K}=\bigcup_{\gamma\in\wGamma} W_{E}^{\ab} w_\gamma $, where $w_\gamma$ are right/left coset representatives of $W_{E}^{\ab}$ in $W_{E/K}$.
     \newline The isomorphism $\phi$ is the following composition: $$\Hom_{\ell\text{-adic}}(T(K),\mql)\to\Hom_{\ell\text{-adic}}(E^{\times},\dlt)_{\wGamma}\xrightarrow[]{\re_{E}^{-1}}\Hom_{\ell\text{-adic}}(W_{E}^{\ab},\dlt)_{\wGamma}\xrightarrow[]{\cor_{\wGamma}}\Hom_{\ell\text{-adic}}(W_{E/K},\dlt).$$
     Consider the following composition,
      $$\Hom_{\ell\text{-adic}}(E^{\times},\dlt)_{\wGamma}\xrightarrow[]{\re_{E}^{-1}}\Hom_{\ell\text{-adic}}(W_{E}^{\ab},\dlt)_{\wGamma}\xrightarrow[]{\cor_{\wGamma}}\Hom_{\ell\text{-adic}}(W_{E/K},\dlt).$$ 
      For any $f\in\Hom_{\ell\text{-adic}}(E^{\times},\dlt)_{\wGamma}$ and $w\in W_{E/K}$, $$\cor_{\wGamma}(r_E^{-1}(f))(w)=\sum_{\gamma\in\wGamma}w_{\gamma}f(r_{E}^{-1}(w_{\gamma}^{-1}w w_{\gamma'}))$$ where $\gamma'$ is such that $w_{\gamma}^{-1}w w_{\gamma'}\in W_{E}^{\ab}$. 
      If the action of $\wGamma$ on $\im(f)$ is trivial,
      \begin{equation*}
         \begin{split}
           \cor_{\wGamma}(r_E^{-1}(f))(w)= & \sum_{\gamma\in\wGamma}w_{\gamma}f(r_{E}^{-1}(w_{\gamma}^{-1}w w_{\gamma'}))\\
             &=f(r_{E}^{-1}( \prod_{\gamma\in\wGamma}w_{\gamma}^{-1}w w_{\gamma'}))\\
             &=f(r_{E}^{-1}\tran(w))\\
             &=f(r_{K}^{-1}(w)).
         \end{split}
      \end{equation*}
    The last step follows from the commutativity of diagram: 
     \begin{center}
     \begin{tikzcd}
     E^{\times}\arrow[r,"r_{E}"]& W_{E}^{\ab}\\
     K^{\times}\arrow[r,"r_{K}"]\arrow[u,hook]& W_{K}^{\ab}\arrow[u,"\tran"]
     \end{tikzcd}
     \end{center}
     Let $\widehat{\kappa}$ denote the map induced on the $\ell$-adic dual by the Kottwitz homomorphism. 
     It now only remains to show that for $f\in\Hom_{\ell\text{-adic}}((\cchg(T)_{\Gamma})^{\Fr},\mql)=\Hom_{\ell\text{-adic}}((\cchg(T)_{\gal_{\K}})^{\Fr},\mql)$, $\wGamma$ acts trivially on $\im(\widehat{\kappa}(f))$.
    Starting with an element $f\in\Hom_{\ell\text{-adic}}((\cchg(T)_{\Gamma})^{\Fr},\mql)\cong \Hom_{\ell\text{-adic}}(\Z,\dlt^{\Gamma})_{\Fr}$ we compute its image in  $\Hom_{\ell\text{-adic}}(T(K),\mql)\cong\Hom_{\ell\text{-adic}}({(E^{\times}})^{\wGamma},\dlt)$ explicitly.  We identify $\dlt$ with $\Hom_{\ell\text{-adic}}(\cchg(T),\mql)$ then for $e\in {(E^{\times})}^{\wGamma}$ and $\lambda\in \cchg(T)$, $$\widehat{\kappa}(f)(e)(\lambda)=f(\lambda^{\val(e)})$$ where $\val$ is the valuation on the field $K$. Observe that then $\im(f)$ has trivial action of $\wGamma$.
    This shows that the first diagram commutes. Thus the following diagram is commutative.
    \begin{center}
     \begin{tikzcd}
     0\arrow[r]&\Hom_{\ell\text{-adic}}((\cchg(T)_{\gal_{\K}})^{\Fr},\mql)\arrow[r,"\widehat{\kappa}"]\arrow[d,"\cong"]&\Hom_{\sm}(T(K),\mql)\arrow[r,"\res"]\arrow[d,"\phi"] & \Hom_{\sm}(\ \T(\Ocal),\mql)\arrow[r]\arrow[d,"\psi"] & 0\\
     0\arrow[r]& H^{1}(\Z,(\dlt)^{\gal_{\K}})\arrow[r] & H^1_{\ell\text{-adic}}(W_K,\dlt)\arrow[r,"\res"]& H^1_{\ell\text{-adic}}(\gal_{\K},\dlt)^{\Fr}\arrow[r] &  0.
     \end{tikzcd}
      \end{center}
 \end{proof}

    \appendix
    
    \section[l-adic Pontryagin duality for abelian profinite groups with Noetherian l- primary part]{$\ell$-adic Pontryagin duality for abelian profinite groups with Noetherian $\ell$- primary part}\label{sec:ladicduality}
    Let $\mathscr{P}_{0}$ denote the category of abelian profinite groups. This category can also be described as the category of zero dimensional commutative pro-algebraic groups. It is an abelian category. Any $A\in \mathscr{P}_{0}$, can be decomposed uniquely as $A=\prod_{p} A_p$, where the product is taken over all prime numbers $p$ and $A_p$ denotes the $p$-primary part of $A$, i.e. all  finite quotients of $A_p$ have order a power of $p$. Note that the $p$-primary part $A_p$ is naturally a $\Z_p$-module. Also, any homomorphism  $A\to B$, in $\mathscr{P}_{0}$ can be decomposed into a product of homomorphisms between the corresponding $p$-primary components (see \cite[\S 4]{Serre1960}). Let $\mathcal{A}b$ denote the category of abelian groups. We have a contravariant  functor to the category of abelian groups, i.e. a functor $$\mathscr{P}_{0}\to\mathcal{A}b^{\opp}$$  $$A\mapsto \Hom(A,\mql),$$ where $\ell$ is a fixed prime number. Since $\mql$ is a divisible group, this functor is exact.
    \par Let ${\mathscr{P}_{0}}_{\ell}$ denote the full subcategory of $\mathscr{P}_{0}$ formed by the abelian profinite groups whose $\ell$-primary part is Noetherian as a $\Z_{\ell}$-module. 
    \par As before $\Hom_{\ell\text{-adic}}(A,\mql)$ consists of the homomorphisms which are continuous for the $\ell$-adic topology on $\mql$. It follows from Lemma \ref{lemma:profinitecon} for any such homomorphism $\rho$,
    $\im\rho \subseteq F$ where $F$ is a finite extension of ${\mathbb{Q}}_{\ell}$. We then have the following
     
    \begin{proposition}\label{cexact}
	The functor $${\mathscr{P}_{0}}_{\ell} \to \mathcal{A}b^{\opp}$$ $$A\mapsto \Hom_{\ell\text{-adic}}(A,\mql)$$ is exact.
	\end{proposition}
	\begin{proof} 
	Let $A\subseteq B$ be objects in ${\mathscr{P}_{0}}_{\ell}$. Given a continuous homomorphism $f:A\to \mql$, we must show that we can extend it to a homomorphism $B\to\mql$. We decompose \(A=\prod_{p\neq \ell} A_{p}\times A_{\ell} \subseteq \prod_{p\neq \ell} B_{p}\times B_{\ell}=B\) into their primary parts. 
    
    Since $A$ is compact $\im(f)\subseteq\Ocal_{F}^{\times}$, where $F$ is a finite extension of $\mathbb{Q}_\ell$ by Lemma \ref{lemma:profinitecon}. Since the co-prime to $\ell$ part of \(\Ocal_{F}^{\times}\) is finite, \(f(\prod_{p\neq \ell}A_{p})\) is finite. Hence $f(A_p)$ must be trivial for all but finitely many primes $p$. For all of these primes, we extend $f|_{A_p}$ to the trivial character of $B_p$. 
    
    This leaves the prime $\ell$, and a finite set $S$ of primes not equal to $\ell$. Since $f(A_p)$ is finite for each $p\in S$, $f_{A_p}$ is in fact a smooth character, which can be extended to a smooth character $B_p\to \mql$ by the exactness of the usual Pontryagin duality for pro-finite groups.  
    
    
    Finally, consider the prime $\ell$. Let $b\in B_{\ell}$ which is a finitely generated $\Z_{\ell}$-module.  Then the $\Z_{\ell}$-submodule $\Z_{\ell}\cdot b  \cap A_{\ell}\subset \Z_{\ell}\cdot b$ is generated by a single element of the form $\ell^nb\in A_{\ell}$. Hence we may extend the character $f|_{\Z_{\ell}\cdot b  \cap A_{\ell}}:\Z_{\ell}\cdot\ell^nb\to \mql$ to all of $\Z_{\ell}\cdot b$. Thus we can extend the character $f:A_{\ell}\to \mql$ to $A_{\ell}+\Z_{\ell}\cdot b$. Using the Noetherian hypothesis, we can extend the character to $B_\ell$. 
    
    Combining the above three steps, we can extend $f:A\to \mql$ to all of $B$. Observe that the character obtained is continuous and the image is contained in a finite extension of $\Q_{\ell}$. 
	\end{proof}

   \section{Lang's Theorem for connected  pro-algebraic groups}\label{sec:langproalg}
Let $k$ be the algebraic closure of a finite field $\Fq$. Let $G$ be a connected (but not necessarily commutative) pro-algebraic group over $k$ equipped with an $\Fq$-structure given by $\Fr:G\to G$. Namely, suppose that $G=\varprojlim\limits_{i\in I} G/H_i$ over a directed set $I$, where each $H_i$ is a $\Fr$-stable normal group subscheme of $G$ such that the quotient $G/H_i$ exists and is a connected quasi-algebraic group. Throughout the paper, we have used the following Lang's Theorem for connected pro-algebraic groups. We include a quick proof of this well known result for completeness.
\begin{theorem}\label{thm:langprofin}
    Let $G$ be a connected pro-algebraic group over $k$ equipped with an $\Fq$-Frobenius $\Fr:G\to G$ as above. Then the Lang map $\Lang:G\to G$ defined by $g\mapsto g\Fr(g^{-1})$ is surjective.
\end{theorem}
\begin{proof}
    Let $h=(h_i)_{i\in I}\in \varprojlim G/H_i=G$. Each $G/H_i$ is connected. By the finite dimensional Lang's theorem, the Lang map $\Lang:G/H_i\to G/H_i$ is surjective, i.e. $\Lang^{-1}(h_i)\subset G/H_i$ is non-empty. Moreover, $\Lang^{-1}(h_i)\subset G/H_i$ is a right $(G/H_i)(\Fq)$-torsor and hence is a finite non-empty set. The sets $(\Lang^{-1}(h_i))_{i\in I}$ form an inverse system of non-empty finite sets. Hence their inverse limit $\varprojlim\Lang^{-1}(h_i)\subset \varprojlim G/H_i$ is also non-empty. This proves that $\Lang:G\to G$ is surjective as desired. 
\end{proof}

\bibliographystyle{plain}
\bibliography{papers}
 \end{document}